\documentclass{amsart}
\usepackage{amsmath,amssymb}

\usepackage[colorlinks=true,linkcolor=blue,citecolor=blue,pdfpagelabels=false]{hyperref}

\theoremstyle{plain}
  \newtheorem{theorem}{Theorem}[section]
  \newtheorem{lemma}[theorem]{Lemma}
  \newtheorem{conjecture}[theorem]{Conjecture}
  \newtheorem{corollary}[theorem]{Corollary}
  \newtheorem{proposition}[theorem]{Proposition}
  
\theoremstyle{definition}
  \newtheorem{example}[theorem]{Example}
  \newtheorem{remark}[theorem]{Remark}
  
\newenvironment{acknowledgements}{\bigskip\textbf{Acknowledgements.}}{}

\newcommand{\mathd}{\mathrm{d}}
\newcommand{\tmem}[1]{{\em #1\/}}
\newcommand{\tmop}[1]{\ensuremath{\operatorname{#1}}}
\newcommand{\tmtextit}[1]{{\itshape{#1}}}
\newcommand{\um}{-}
\newcommand{\upl}{+}
\newcommand{\upm}{\pm}

\newcommand{\mathbbm}[1]{\mathbb{#1}}

\newcommand{\assign}{:=}

\newcommand{\sumprime}{\mathop{\sum\nolimits'}}
\renewcommand{\pmod}[1]{\,(\textup{mod}\,#1)}

\title[On a secant Dirichlet series]%
{On a secant Dirichlet series and Eichler integrals of Eisenstein series}


\author{Bruce C. Berndt}
\thanks{The first author's research was partially supported by NSA grant
H98230-11-1-0200.}
\address{Department of Mathematics,
University of Illinois at Urbana-Champaign,
1409 W. Green St,
Urbana, IL 61801,
United States}
\email{berndt@illinois.edu}

\author{Armin Straub}
\address{Department of Mathematics,
University of Illinois at Urbana-Champaign,
1409 W. Green St,
Urbana, IL 61801,
United States}
\curraddr{Max-Planck-Institut f\"ur Mathematik,
Vivatsgasse 7,
53111 Bonn,
Germany}
\email{astraub@illinois.edu}

\date{June 7, 2014}

\begin{document}
\maketitle

\begin{abstract}
  We consider, for even $s$, the secant Dirichlet series $\psi_s (\tau) =
  \sum_{n = 1}^{\infty} \frac{\sec (\pi n \tau)}{n^s}$, recently introduced
  and studied by Lal\'{\i}n, Rodrigue and Rogers. In particular, we show, as
  conjectured and partially proven by Lal\'{\i}n, Rodrigue and Rogers, that
  the values $\psi_{2 m} ( \sqrt{r})$, with $r > 0$ rational, are rational
  multiples of $\pi^{2 m}$. We then put the properties of the secant Dirichlet
  series into context by showing that they are Eichler integrals of odd weight
  Eisenstein series of level $4$. This leads us to consider Eichler integrals
  of general Eisenstein series and to determine their period polynomials. In
  the level $1$ case, these polynomials were recently shown by Murty, Smyth
  and Wang to have most of their roots on the unit circle. We provide evidence
  that this phenomenon extends to the higher level case. This observation
  complements recent results by Conrey, Farmer and Imamoglu as well as
  El-Guindy and Raji on zeros of period polynomials of Hecke eigenforms in the
  level $1$ case. Finally, we briefly revisit results of a similar type in the
  works of Ramanujan.
\end{abstract}

\section{Introduction}

Our considerations begin with the secant Dirichlet series
\begin{equation}
  \psi_s (\tau) = \sum_{n = 1}^{\infty} \frac{\sec (\pi n \tau)}{n^s},
  \label{eq:defsz}
\end{equation}
which were recently introduced and studied by Lal\'{\i}n, Rodrigue and Rogers
{\cite{lrr-secantzeta}}. One of the motivations for considering these sums is
their similarity in shape and, as we will see, in properties to the cotangent
Dirichlet series
\begin{equation}
  \xi_s (\tau) = \sum_{n = 1}^{\infty} \frac{\cot (\pi n \tau)}{n^s} .
  \label{eq:defcz}
\end{equation}
For instance, as first proved by Lerch and also recorded by Ramanujan (see
Section \ref{sec:history} as well as {\cite{berndt-modular77}} or {\cite[p.
276]{berndtII}}) the difference $\tau^{2 m} \xi_{2 m + 1} (- 1 / \tau) -
\xi_{2 m + 1} (\tau)$ is a rational function in $\tau$. This modular
functional equation and its ramifications continue to inspire research to this
date, including, for instance, {\cite{gun-murty-rath2011}},
{\cite{msw-zeroes}}, {\cite{rivoal-conv12}}, {\cite{ls-roots12}}.

As shown in {\cite{lrr-secantzeta}}, the secant Dirichlet series $\psi_{2 m}
(\tau)$ satisfy modular functional equations as well. In Section
\ref{sec:funres}, we give an alternative derivation of these functional
equations based on the residue theorem and in the spirit of
{\cite{berndt-dedekindsum76}}. In this way, we obtain a compact representation
of the associated rational function as a certain Taylor coefficient of a
quotient of trigonometric functions. We then show in Section \ref{sec:sqrt},
as conjectured and partially proven in {\cite{lrr-secantzeta}}, that, for a
positive rational number $r$, $\psi_{2 m} ( \sqrt{r})$ is a rational multiple
of $\pi^{2 m}$.

In Section \ref{sec:eisenstein}, we observe that the $2 m$th derivative of
$\psi_{2 m} (2 \tau)$ is, up to a constant term, an Eisenstein series of
weight $2 m + 1$, level $4$ and character $\chi_{- 4}$. In other words, the
secant Dirichlet series $\psi_{2 m}$ are Eichler integrals of Eisenstein
series; the basic theory of Eichler integrals is reviewed in Section
\ref{sec:eichler}. In this light, several of the properties of the function
$\psi_{2 m}$ become natural and, in Section \ref{sec:periods}, we use the
modular setting to give another derivation of the functional equation by
evaluating the period polynomial of the corresponding Eisenstein series.

In fact, the computation of the period polynomials in Section
\ref{sec:periods} is carried out more generally for Eisenstein series
corresponding to pairs of Dirichlet characters. A special case of this
computation allows us, for instance, to rederive the Ramanujan-style formulas
for Dirichlet $L$-values of {\cite{katayama-L74}}. The resulting period
polynomials mirror the well-known polynomials occuring in the level $1$ case,
studied, for instance, in {\cite{gun-murty-rath2011}} and {\cite{msw-zeroes}},
where they (more accurately, their odd parts) are referred to as the
{\tmem{Ramanujan polynomials}}. In {\cite{msw-zeroes}} it was shown that the
Ramanujan polynomials are nearly unimodular, that is, all their nonreal roots
lie on the unit circle. On the other hand, it was conjectured in
{\cite{lr-unimod13}}, and proven in {\cite{ls-roots12}}, that the full period
polynomial is unimodular in the level $1$ case, and this property is also
shown to hold for several related polynomials. We indicate in Section
\ref{sec:rpx} that, after a linear change of variables, the period polynomials
in higher level appear to have all or most of their roots on the unit circle
as well. This observation fits well with and complements the recent result of
{\cite{cfi-ppz}} and {\cite{er-ppz13}}, where it is shown that the nontrivial
zeros of period polynomials of modular forms, which are Hecke eigenforms of
level $1$, all lie on the unit circle. As an application, we derive formulas
for Dirichlet $L$-values in terms of values of an Eichler integral at
algebraic arguments of modulus $1$, thus generalizing the formulas for $\zeta
(2 k + 1)$ studied in {\cite{gun-murty-rath2011}}.

Finally, in Section \ref{sec:history}, we return to and discuss related
entries in Ramanujan's notebook {\cite{berndtII}}, including the particularly
famous entry corresponding to the cotangent Dirichlet series (\ref{eq:defcz}),
which, for odd $s$, are the Eichler integrals of the Eisenstein series for the
full modular group. In particular, we close by demonstrating that the
functional equation for $\psi_{2 m} (\tau)$ is in fact a consequence of an
identity stated by Ramanujan.

It is on purpose, and hopefully to the benefit of some readers, that the
considerations in this paper start out entirely explicit and elementary, with
theoretical background, such as Eisenstein series of higher level and Eichler
integrals, being included as we proceed. As a consequence of this approach,
some of the earlier results can be obtained as special cases of later results.

Convergence of series such as (\ref{eq:defsz}) and (\ref{eq:defcz}), when
$\tau$ is a real number, is a rather subtle issue; see, for instance,
{\cite{rivoal-conv12}}. It is shown in {\cite{lrr-secantzeta}} that $\psi_s
(\tau)$, for $s \geqslant 2$, converges absolutely for rational $\tau$ with
odd denominator as well as for algebraic irrational $\tau$. On the other hand,
the series converges absolutely for all nonreal $\tau$, and our discussion of
the modular properties proceeds under the tacit assumption that $\tau$ is not
real. In the case of evaluations of $\psi_{2 m} (\tau)$ at real quadratic
$\tau$ in Section \ref{sec:sqrt}, one may then use limiting arguments to show
that the functional equations also hold for these arguments.

\section{The functional equation of the secant Dirichlet series via
residues}\label{sec:funres}

Obviously, $\psi_s (\tau)$ is periodic of period $2$, that is $\psi_s (\tau +
2) = \psi_s (\tau)$. In this section, we present an alternative proof of the
additional nontrivial functional equation satisfied by (\ref{eq:defsz}) in the
case $s = 2 m$. In {\cite{lrr-secantzeta}} this functional equation was
established by skillfully and carefully convoluting partial fraction
expansions. The proof given here is based on the residue theorem and is in the
spirit of the proof given in {\cite{berndt-dedekindsum76}} for the cotangent
Dirichlet series (\ref{eq:defcz}). Throughout, we denote with $[z^n] f (z)$
the $n$th coefficient of the Taylor expansion of $f (z)$.

\begin{theorem}
  \label{thm:fun}Let $m > 0$ be an integer. Then we have
  \begin{eqnarray}
    &  & \left( 1 + \tau \right)^{2 m - 1} \psi_{2 m} \left( \frac{\tau}{1 +
    \tau} \right) - \left( 1 - \tau \right)^{2 m - 1} \psi_{2 m} \left(
    \frac{\tau}{1 - \tau} \right) \nonumber\\
    & = & \pi^{2 m}  \left[ z^{2 m - 1} \right]  \frac{\sin (\tau z)}{\sin (
    \left( 1 - \tau \right) z) \sin ( \left( 1 + \tau \right) z)} . 
    \label{eq:phifunsym}
  \end{eqnarray}
\end{theorem}

\begin{proof}
  Let $s > 0$ be an integer and consider
  \begin{equation}
    I_N = \frac{1}{2 \pi i} \int_{C_N} \frac{\sin \left( \pi \frac{b - a}{2} z
    \right)}{\sin (\pi a z) \sin (\pi b z)}  \frac{\mathd z}{z^{s + 1}},
    \label{eq:IN}
  \end{equation}
  where $C_N$ is a positively oriented circle of radius $R_N$ centered at the
  origin. As in {\cite{berndt-dedekindsum76}}, the radii $R_N$ are chosen such
  that the points on the circle are always at least some fixed positive
  distance from any of the points $\frac{n}{a}$ and $\frac{n}{b}$, where $n$
  ranges over the integers. It is then easily seen that
  \begin{equation}
    \lim_{N \rightarrow \infty} I_N = 0. \label{eq:IN0}
  \end{equation}
  The integrand of (\ref{eq:IN}) has poles at $z = 0$ as well as at $z =
  \frac{n}{a}$ and $z = \frac{n}{b}$ for $n \in \mathbbm{Z}$. Writing
  $\tmop{Res} (\alpha)$ for the residue of the integrand at $z = \alpha$, we
  have
  \[ \tmop{Res} \left( \frac{n}{a} \right) = a^s  \frac{(- 1)^n}{\pi} 
     \frac{1}{n^{s + 1}}  \frac{\sin \left( \pi n \frac{b - a}{2 a}
     \right)}{\sin \left( \pi n \frac{b}{a} \right)} = \frac{a^s}{2 \pi} 
     \frac{\sec \left( \pi n \frac{b - a}{2 a} \right)}{n^{s + 1}}, \]
  where the last equality is obtained by writing $(- 1)^n = \frac{1}{\cos (\pi
  n)}$ and using the trigonometric identity
  \[ \cos (x) \sin (y) - \cos (y) \sin (x) = 2 \cos \left( \frac{y - x}{2}
     \right) \sin \left( \frac{y - x}{2} \right), \]
  with $x = \pi n$ and $y = \pi n \frac{b}{a}$. By symmetry,
  \[ \tmop{Res} \left( \frac{n}{b} \right) = - \frac{b^s}{2 \pi}  \frac{\sec
     \left( \pi n \frac{b - a}{2 b} \right)}{n^{s + 1}} . \]
  For the residue at the origin, we find that
  \[ \tmop{Res} (0) = \left[ z^s \right]  \frac{\sin \left( \pi \frac{b -
     a}{2} z \right)}{\sin (\pi a z) \sin (\pi b z)} = \pi^s \left[ z^s
     \right]  \frac{\sin \left( \frac{b - a}{2} z \right)}{\sin (a z) \sin (b
     z)} = : \pi^s p_s (a, b) . \]
  We now assume that $s$ is odd. Collecting residues and letting $N
  \rightarrow \infty$, we obtain, using (\ref{eq:IN0}),
  \[ 0 = \frac{a^s}{\pi} \sum_{n = 1}^{\infty} \frac{\sec \left( \pi n \frac{b
     - a}{2 a} \right)}{n^{s + 1}} - \frac{b^s}{\pi} \sum_{n = 1}^{\infty}
     \frac{\sec \left( \pi n \frac{b - a}{2 b} \right)}{n^{s + 1}} + \pi^s p_s
     (a, b) . \]
  Finally, let $a = 1 - \tau$ and $b = 1 + \tau$, and write $s = 2 m - 1$, to
  obtain the functional equation in the form
  \[ (1 + \tau)^{2 m - 1} \psi_{2 m} \left( \frac{\tau}{1 + \tau} \right) - (1
     - \tau)^{2 m - 1} \psi_{2 m} \left( \frac{\tau}{1 - \tau} \right) =
     \pi^{2 m} p_{2 m - 1} (1 - \tau, 1 + \tau), \]
  as claimed.
\end{proof}

Note that, upon replacing $\tau \rightarrow \frac{\tau}{\tau + 1}$ and
simplifying slightly, we find that equation (\ref{eq:phifunsym}) becomes
\begin{equation}
  (2 \tau + 1)^{2 m - 1} \psi_{2 m} \left( \frac{\tau}{2 \tau + 1} \right) -
  \psi_{2 m} (\tau) = \pi^{2 m}  \left[ z^{2 m - 1} \right]  \frac{\sin (\tau
  z)}{\sin (z) \sin ((2 \tau + 1) z)}, \label{eq:phifun}
\end{equation}
which, while less symmetric, makes the modular transformation property more
apparent.

\begin{remark}
  The rational functions on the right-hand sides of (\ref{eq:phifunsym}) and
  (\ref{eq:phifun}) can be made more explicit by expanding the sine functions.
  Namely, if we expand the two sine functions in the denominator using the
  defining generating function,
  \begin{equation}
    \frac{z e^{x z}}{e^z - 1} = \sum_{n \geqslant 0} B_n (x) \frac{z^n}{n!},
    \label{eq:B}
  \end{equation}
  of the Bernoulli polynomials $B_n (x)$, then the rational functions are seen
  to be equal to double sums such as
  \[ \left[ z^{2 m - 1} \right]  \frac{\sin \left( \frac{b - a}{2} z
     \right)}{\sin (a z) \sin (b z)} = \frac{(- 1)^m}{a b} \sum_{k + n + r =
     m} \frac{\left( \frac{b - a}{2} \right)^{2 k + 1} B_{2 n} ( \tfrac{1}{2})
     B_{2 r} ( \tfrac{1}{2}) (2 a)^{2 n} (2 b)^{2 r}}{(2 k + 1) ! (2 n) ! (2
     r) !}, \]
  where the sum is over nonnegative integers $k, n, r$ summing to $m$. We
  remark that $B_n ( \tfrac{1}{2}) = - (1 - 2^{1 - n}) B_n$, where $B_n = B_n
  (0)$ is the $n$th Bernoulli number. Note that this shows, in particular,
  that the right-hand sides of equations (\ref{eq:phifunsym}) and
  (\ref{eq:phifun}) are rational functions of the form $\pi^{2 m} p_{2 m}
  (\tau) / (1 - \tau^2)$ and $\pi^{2 m} q_{2 m} (\tau) / (2 \tau + 1)$,
  respectively, where $p_{2 m}$ and $q_{2 m}$ are polynomials of degree $2 m +
  1$ with rational coefficients.
  
  In fact, as we will see from versions of the functional equation derived
  later in this paper, the mentioned double sums can be reduced to single
  summations; see Remark \ref{rk:phifun-12}.
\end{remark}

\section{The secant Dirichlet series at real quadratic
irrationals}\label{sec:sqrt}

In {\cite{lrr-secantzeta}} it was conjectured and partially proven that
$\psi_{2 m} ( \sqrt{k})$ is a rational multiple of $\pi^{2 m}$ whenever $m$
and $k$ are positive integers. In this section, we prove this assertion and
extend it to the case when $k$ is a rational number. For the sake of
simplicity, we first prove, in Theorem \ref{thm:phisqrtk}, the case of
integral $k$ and then briefly indicate how the more general case, stated in
Theorem \ref{thm:phisqrtr}, follows in essentially the same way. We point out
that an independent proof, based on the theory of generalized
$\eta$-functions, of Theorem~\ref{thm:phisqrtr} has been given by
P.~Charollois and M.~Greenberg {\cite{cg-secant}}. We thank M.~Lal\'{\i}n, who
received preprints of both {\cite{cg-secant}} and this paper within a couple
of days, for making us aware of this reference.

Consider the matrices
\begin{equation}
  A = \left(\begin{array}{cc}
    1 & 2\\
    0 & 1
  \end{array}\right), \hspace{1em} B = \left(\begin{array}{cc}
    1 & 0\\
    2 & 1
  \end{array}\right) . \label{eq:AB}
\end{equation}
As usual, we denote by $\Gamma (2)$ the congruence subgroup
\[ \Gamma (2) = \left\{ \gamma \in \tmop{SL}_2 (\mathbbm{Z}) : \hspace{1em}
   \gamma \equiv I \pmod{2} \right\}, \]
where $I$ is the identity matrix. The group $\Gamma (2)$ is generated by $A$,
$B$ and $- I$; see, for instance, {\cite[Exercise II.6, p.34]{yoshida-hyp}}.

\begin{remark}
  We note that, as is well-known, $\Gamma (2)$ is conjugate to $\Gamma_0 (4)$
  under $\tau \mapsto 2 \tau$, and that in this setting the group $\langle A,
  B \rangle$ corresponds to $\Gamma_1 (4)$.
\end{remark}

In terms of the matrices $A$ and $B$, we then have
\begin{equation}
  \psi_{2 m} (A \tau) = \psi_{2 m} (\tau), \hspace{1em} \psi_{2 m} (B \tau) =
  \frac{1}{(2 \tau + 1)^{2 m - 1}} \psi_{2 m} (\tau) + \pi^{2 m} f_{2 m}
  (\tau) \label{eq:funAB},
\end{equation}
where $f_{2 m} (\tau)$, given in (\ref{eq:phifun}), is a rational function
over $\mathbbm{Q}$.

Note that the first equation in (\ref{eq:funAB}) simply expresses the
periodicity of $\psi_{2 m}$, that is $\psi_{2 m} (\tau + 2) = \psi_{2 m}
(\tau)$, while the second one is the functional equation (\ref{eq:phifun})
discussed in the previous section.

\begin{theorem}
  \label{thm:phisqrtk}Let $m$ and $k$ be positive integers. Then
  \[ \psi_{2 m} ( \sqrt{k}) \in \pi^{2 m} \mathbbm{Q}. \]
\end{theorem}

\begin{proof}
  Following {\cite{lrr-secantzeta}}, we observe that if the integers $X$ and
  $Y$ solve Pell's equation
  \begin{equation}
    X^2 - k Y^2 = 1, \label{eq:pell}
  \end{equation}
  then $D \cdot \sqrt{k} = \sqrt{k}$, where $D$ is the matrix
  \[ D = \left(\begin{array}{cc}
       X & k Y\\
       Y & X
     \end{array}\right) \in \tmop{SL}_2 (\mathbbm{Z}) . \]
  Here, and in the sequel, we let $2 \times 2$ matrices act on complex numbers
  by linear fractional transformations and write
  \[ \left(\begin{array}{cc}
       a & b\\
       c & d
     \end{array}\right) \cdot x = \frac{a x + b}{c x + d} . \]
  A proof that for every positive nonsquare $k$ there exist nontrivial
  solutions $X$, $Y$ to Pell's equation (\ref{eq:pell}) was first published by
  Lagrange in 1768 {\cite{lagrange-pell}}. For further information and
  background on Pell's equation we refer to {\cite{lenstra-pell}}.
  
  We now make the simple observation that, by (\ref{eq:pell}),
  \[ D^2 = \left(\begin{array}{cc}
       X & k Y\\
       Y & X
     \end{array}\right)^2 = \left(\begin{array}{cc}
       X^2 + k Y^2 & 2 k X Y\\
       2 X Y & X^2 + k Y^2
     \end{array}\right) \in \Gamma (2) . \]
  Hence, we can always find $C \in \left\langle A, B, - I \right\rangle =
  \Gamma (2)$, $C \neq \pm I$, such that $C \cdot \sqrt{k} = \sqrt{k}$. Let $C
  = \left(\begin{array}{cc}
    a & b\\
    c & d
  \end{array}\right)$ be such a matrix. Note that necessarily $c \neq 0$.
  
  Repeatedly applying the functional equations (\ref{eq:funAB}), we deduce
  that
  \[ \psi_{2 m} ( \sqrt{k}) = \psi_{2 m} (C \cdot \sqrt{k}) = \frac{1}{(c
     \sqrt{k} + d)^{2 m - 1}} \psi_{2 m} ( \sqrt{k}) + \pi^{2 m} f_{2 m, C} (
     \sqrt{k}), \]
  where $f_{2 m, C} (\tau) \in \mathbbm{Q} (\tau)$ is a rational function with
  rational coefficients. Consequently,
  \[ \psi_{2 m} ( \sqrt{k}) = \pi^{2 m} g_k ( \sqrt{k}) \]
  for some rational function $g_k (\tau) \in \mathbbm{Q} (\tau)$. On the other
  hand, the above argument applied to $- \sqrt{k}$ shows that $\psi_{2 m} (-
  \sqrt{k}) = \pi^{2 m} g_k (- \sqrt{k})$. Since $\psi_{2 m}$ is even, it
  follows that
  \[ g_k ( \sqrt{k}) = g_k (- \sqrt{k}), \]
  which implies that $g_k ( \sqrt{k})$ is a rational number.
\end{proof}

\begin{example}
  We now illustrate Theorem \ref{thm:phisqrtk} by evaluating $\psi_2 (
  \sqrt{2})$. In this case, the fundamental solution to Pell's equation
  (\ref{eq:pell}) is given by $(X, Y) = (3, 2)$, obtained from the first two
  terms of the continued fraction of $\sqrt{2}$. Then
  \[ C = \left(\begin{array}{cc}
       3 & 4\\
       2 & 3
     \end{array}\right) = - A B^{- 1} A, \]
  and, indeed, one easily verifies that $\sqrt{2}$ is fixed under $A B^{- 1}
  A$. In the present case $s = 1$, and the transformations (\ref{eq:funAB})
  satisfied by $\psi_2$ take the form
  \begin{equation}
    \psi_2 (A \tau) = \psi_2 (\tau), \hspace{1em} \psi_2 (B \tau) = \frac{1}{2
    \tau + 1} \psi_2 (\tau) + \pi^2 f_2 (\tau), \label{eq:fun2AB}
  \end{equation}
  with
  \[ f_2 (\tau) = \frac{\tau (3 \tau^2 + 4 \tau + 2)}{6 (2 \tau + 1)^2} . \]
  We therefore find that
  \begin{eqnarray*}
    \psi_2 (A B^{- 1} A \tau) & = & \psi_2 (B^{- 1} A \tau)\\
    & = & (2 B^{- 1} A \tau + 1) \left[ \psi_2 (A \tau) - \pi^2 f_2 (B^{- 1}
    A \tau) \right]\\
    & = & - \frac{1}{2 \tau + 3} \psi_2 (\tau) - \frac{(\tau + 2) (3 \tau^2 +
    8 \tau + 6)}{6 (2 \tau + 3)^2} \pi^2 .
  \end{eqnarray*}
  For the second equality we applied the second transformation of
  (\ref{eq:fun2AB}) with $B^{- 1} A \tau$ in place of $\tau$, while for the
  third equality we use the fact that $B^{- 1} A \tau = - \frac{\tau + 2}{2
  \tau + 3}$.
  
  For $\tau = \sqrt{2}$ this reduces to
  \[ \psi_2 ( \sqrt{2}) = (2 \sqrt{2} - 3) \psi_2 ( \sqrt{2}) + \frac{2}{3} (
     \sqrt{2} - 2) \pi^2, \]
  which has the solution $\psi_2 ( \sqrt{2}) = - \frac{\pi^2}{3}$, in
  agreement with the value given in {\cite{lrr-secantzeta}}. Families of more
  general explicit evaluations of $\psi_{2 m}$ at real quadratic
  irrationalities are derived in Example \ref{eg:phisqrtvals}.
\end{example}

In fact, it is the case that $\psi_{2 m} ( \sqrt{r})$ is a rational multiple
of $\pi^{2 m}$ whenever $r$ is a rational number. This may be shown in
essentially the same fashion, as we indicate next.

\begin{theorem}
  \label{thm:phisqrtr}Let $m$ be a positive integer and $r$ a positive
  rational number such that $\sqrt{r}$ is not a rational number with an even
  denominator. Then
  \[ \psi_{2 m} ( \sqrt{r}) \in \pi^{2 m} \mathbbm{Q}. \]
\end{theorem}

\begin{proof}
  Write $r = \frac{a}{b}$, where $a$ and $b$ are positive integers. Similarly
  to the proof of Theorem \ref{thm:phisqrtk}, we observe that if the integers
  $X$ and $Y$ solve Pell's equation $X^2 - a b Y^2 = 1$, then $D \cdot
  \sqrt{r} = \sqrt{r}$, where $D$ is the matrix
  \[ D = \left(\begin{array}{cc}
       X & a Y\\
       b Y & X
     \end{array}\right) \in \tmop{SL}_2 \left( \mathbbm{Z} \right) . \]
  Again, one observes that $D^2 \in \Gamma (2)$. Hence we can find $C \in
  \left\langle A, B, - I \right\rangle = \Gamma (2)$, $C \neq \pm I$, such
  that $C \cdot \sqrt{r} = \sqrt{r}$.
  
  It remains to proceed exactly as in the proof of Theorem \ref{thm:phisqrtk}.
\end{proof}

\begin{remark}
  \label{rk:phisqrtr}We remark that (parts of) the proofs of Theorems
  \ref{thm:phisqrtk} and \ref{thm:phisqrtr} also apply to other functions that
  satisfy a modular transformation which is of a similar form as the one for
  the secant Dirichlet series $\psi_{2 m}$. This includes the class of Eichler
  integrals to which, as discussed in the Section \ref{sec:eichler}, $\psi_{2
  m}$ belongs.
\end{remark}

\section{Eisenstein series}\label{sec:eisenstein}

We begin this section by observing a natural relation between the secant
Dirichlet series $\psi_{2 m}$ and Eisenstein series. Specifically, we note
that the $2 m$th derivative of $\psi_{2 m}$ is, essentially, an Eisenstein
series of weight $2 m + 1$ and level $4$. In the language of Section
\ref{sec:eichler}, this means that the secant Dirichlet series are Eichler
integrals of certain Eisenstein series. In order to investigate Eichler
integrals of Eisenstein series in general, we recall basic facts about
Eisenstein series in this section.

Throughout, we write $D = \frac{\mathd}{\mathd \tau}$ and $q = e^{2 \pi i
\tau}$. For $n \geqslant 0$, we denote with $E_n$ the $n$th Euler number,
defined by
\begin{equation}
  \tmop{sech} x = \sum_{n = 0}^{\infty} E_n  \frac{x^n}{n!}, \label{eq:defen}
\end{equation}
where $|x| < \pi$.

\begin{lemma}
  \label{lem:phiD}We have
  \[ D^{2 m} [\psi_{2 m} (\tau / 2)] = \frac{(2 m) !}{\pi} \sumprime_{k, j \in
     \mathbbm{Z}} \frac{\chi_{- 4} (j)}{(k \tau + j)^{2 m + 1}} - \frac{(-
     1)^m E_{2 m} \pi^{2 m}}{2^{2 m + 1}}, \]
  where $\chi_{- 4} = ( \tfrac{- 4}{\cdot})$ is the nonprincipal Dirichlet
  character modulo $4$ (that is, $\chi_{- 4} (n) = 0$ for even $n$, and
  $\chi_{- 4} (n) = (- 1)^{(n - 1) / 2}$ for odd $n$).
\end{lemma}

\begin{proof}
  In light of the partial fraction expansion of the secant function
  \[ \sec \left(  \frac{\pi \tau}{2} \right) = \frac{4}{\pi} \sum_{j \geqslant
     1} \frac{\chi_{- 4} (j) j}{j^2 - \tau^2} = \lim_{N \rightarrow \infty}
     \frac{2}{\pi} \sum_{j = - N}^N \frac{\chi_{- 4} (j)}{\tau + j}, \]
  we derive that
  \[ D^k \sec \left( \frac{\pi \tau}{2} \right) = \frac{2 (- 1)^k k!}{\pi}
     \sum_{j \in \mathbbm{Z}} \frac{\chi_{- 4} (j)}{(\tau + j)^{k + 1}} . \]
  Consequently,
  \begin{eqnarray}
    D^{2 m} \sum_{n \geqslant 1} \frac{\sec \left( \frac{\pi n \tau}{2}
    \right)}{n^{2 m}} & = & \frac{2 (2 m) !}{\pi} \sum_{k \geqslant 1} \sum_{j
    \in \mathbbm{Z}} \frac{\chi_{- 4} (j)}{(k \tau + j)^{2 m + 1}} \nonumber\\
    & = & \frac{(2 m) !}{\pi} \sumprime_{k, j \in \mathbbm{Z}} \frac{\chi_{- 4}
    (j)}{(k \tau + j)^{2 m + 1}} - \frac{2 (2 m) !}{\pi} L (\chi_{- 4}, 2 m +
    1),  \label{eq:phiDL}
  \end{eqnarray}
  where
  \[ L (\chi_{- 4}, s) = \sum_{n = 1}^{\infty} \frac{\chi_{- 4} (n)}{n^s} =
     \sum_{n = 0}^{\infty} \frac{(- 1)^n}{(2 n + 1)^s} \]
  is the Dirichlet $L$-series attached to $\chi_{- 4}$, also known as the
  Dirichlet beta function. Using Euler's well-known evaluation {\cite{ayoub1}}
  \begin{equation}
    L (\chi_{- 4}, 2 m + 1) = \frac{1}{2}  \frac{\left( - 1 \right)^m E_{2
    m}}{(2 m) !}  \left( \frac{\pi}{2} \right)^{2 m + 1} \label{eq:L4odd}
  \end{equation}
  in (\ref{eq:phiDL}), we complete the proof of Lemma \ref{lem:phiD}.
\end{proof}

\begin{example}
  In the case $m = 1$, we find that
  \[ D^2 \psi_2 (\tau / 2) = \frac{2}{\pi} \sumprime_{k, j \in \mathbbm{Z}}
     \frac{\chi_{- 4} (j)}{(k \tau + j)^3} - \frac{\pi^2}{8}, \]
  since $L (\chi_{- 4}, 3) = \frac{\pi^3}{32}$. The Eisenstein series
  \[ \sumprime_{k, j \in \mathbbm{Z}} \frac{\chi_{- 4} (j)}{(4 k \tau + j)^3} =
     \frac{\pi^3}{4} \left[ \frac{1}{4} - q - q^2 + 8 q^3 - q^4 - 26 q^5 +
     \cdots \right] \]
  is a modular form of weight $3$, level $4$ and character $\chi_{- 4}$.
\end{example}

Define, as in {\cite[Chapter 7]{miyake}}, the Eisenstein series
\begin{equation}
  E_k (\tau ; \chi, \psi) = \sumprime_{m, n \in \mathbbm{Z}} \frac{\chi (m) \psi
  (n)}{(m \tau + n)^k}, \label{eq:eisenstein}
\end{equation}
where $k > 2$, and $\chi$ and $\psi$ are Dirichlet characters modulo $L$ and
$M$, respectively. As detailed in {\cite[Chapter 7]{miyake}}, these Eisenstein
series can be used to generate all Eisenstein series with respect to any
congruence subgroup.

\begin{example}
  \label{eg:phiDE}By Lemma \ref{lem:phiD}, the secant Dirichlet series $\psi$
  is connected with the case $\psi = \chi_{- 4}$ and $\chi = 1$, the principal
  character modulo $1$. To be precise,
  \begin{equation}
    D^{2 m} [\psi_{2 m} (\tau / 2)] = \frac{(2 m) !}{\pi} \left[ E_{2 m + 1}
    (\tau ; 1, \chi_{- 4}) - E_{2 m + 1} (i \infty ; 1, \chi_{- 4}) \right] .
    \label{eq:phiDE}
  \end{equation}
  That the constant term on the right-hand side indeed agrees with the one
  stated in Lemma \ref{lem:phiD} will become clear from the facts about the
  Eisenstein series $E_k (\tau ; \chi, \psi)$ which we review next.
\end{example}

\begin{example}
  As detailed in Section \ref{sec:rama1}, the cotangent Dirichlet series
  $\xi_s (\tau)$, introduced in (\ref{eq:defcz}), is in a similar way related
  to the Eisenstein series $E_{2 m} (\tau ; 1, 1)$, which is modular with
  respect to the full modular group.
\end{example}

In the sequel, we always assume that $\chi (- 1) \psi (- 1) = (- 1)^k$ since,
otherwise, $E_k (\tau ; \chi, \psi) = 0$. In order to derive the Fourier
expansion of the Eisenstein series, we recall that, by the character analogue
of the Lipschitz summation formula {\cite{berndt-char}}, for any primitive
Dirichlet character $\psi$ of modulus $M$,
\[ \sum_{n = - \infty}^{\infty} \frac{\psi (n)}{(\tau + n)^s} = G (\psi)
   \frac{(- 2 \pi i / M)^s}{\Gamma (s)} \sum_{m = 1}^{\infty} \bar{\psi} (m)
   m^{s - 1} e^{2 \pi i m \tau / M} . \]
Here, and in the sequel, $G (\psi) = \sum_{a = 1}^M \psi (a) e^{2 \pi i a /
M}$ denotes the Gauss sum associated with $\psi$. If $\psi$ is primitive, we
thus find that
\begin{eqnarray}
  E_k (\tau ; \chi, \psi) & = & a_0 (E_k) + 2 \sum_{m = 1}^{\infty} \chi (m)
  \sum_{n \in \mathbbm{Z}} \frac{\psi (n)}{(m \tau + n)^k} \nonumber\\
  & = & a_0 (E_k) + 2 G (\psi) \frac{(- 2 \pi i / M)^k}{\Gamma (k)} \sum_{m =
  1}^{\infty} \chi (m) \sum_{n = 1}^{\infty} \bar{\psi} (n) n^{k - 1} e^{2 \pi
  i n m \tau / M} \nonumber\\
  & = & a_0 (E_k) + A \sum_{n = 1}^{\infty} a_n (E_k) e^{2 \pi i n \tau / M},
  \label{eq:Eqexp}
\end{eqnarray}
where $A = 2 G (\psi) (- 2 \pi i / M)^k / \Gamma (k)$ and
\[ a_0 (E_k) = \left\{ \begin{array}{ll}
     2 L (k, \psi), & \text{if $\chi = 1$,}\\
     0, & \text{otherwise,}
   \end{array} \right. \hspace{2em} a_n (E_k) = \sum_{d|n} \chi (n / d)
   \bar{\psi} (d) d^{k - 1} . \]
We note that the series in (\ref{eq:Eqexp}) converges for any complex value of
$k$, and hence (\ref{eq:Eqexp}) provides the analytic continuation of
(\ref{eq:eisenstein}) to the entire complex $k$-plane.

\begin{example}
  In light of Example \ref{eg:phiDE}, we therefore find that the $q$-expansion
  of the secant Dirichlet series is given by
  \begin{equation}
    \psi_{2 m} (2 \tau) = 2 \sum_{n \geqslant 1} \left( \sum_{d|n} \chi_{- 4}
    (d) d^{2 m} \right) \frac{q^n}{n^{2 m}} . \label{eq:phiqexp}
  \end{equation}
\end{example}

Recall that the $L$-function of a modular form $f (\tau) = \sum_{n =
0}^{\infty} b (n) e^{2 \pi i n \tau / \lambda}$ is defined as
\begin{equation}
  L (f, s) = \frac{(2 \pi)^s}{\Gamma (s)} \int_0^{\infty} \left[ f \left( i
  \tau \right) - f (i \infty) \right] \tau^{s - 1} \mathd \tau = \lambda^s
  \sum_{n = 1}^{\infty} \frac{b (n)}{n^s} . \label{eq:Lf}
\end{equation}
As another consequence of (\ref{eq:Eqexp}), if $\psi$ is primitive, the
$L$-function of $E (\tau) = E_k (\tau ; \chi, \psi)$ is given by
\begin{equation}
  L (E, s) = A M^s L (\chi, s) L ( \bar{\psi}, 1 - k + s) . \label{eq:LE}
\end{equation}
Depending on the parity of $s$, the values of $L (E, s)$ can be evaluated in
terms of generalized Bernoulli numbers $B_{n, \chi}$, which are defined by
\begin{equation}
  \sum_{n = 0}^{\infty} B_{n, \chi}  \frac{x^n}{n!} = \sum_{a = 1}^L
  \frac{\chi (a) x e^{a x}}{e^{L x} - 1}, \label{eq:Bchi}
\end{equation}
if $\chi$ is a Dirichlet character modulo $L$. We observe that, for $n \neq
1$, the classical Bernoulli numbers $B_n$ are equal to $B_{n, \chi}$ with
$\chi = 1$. Similarly, the Euler numbers are connected with the case $\chi =
\chi_{- 4}$:
\begin{equation}
  \frac{1}{2}  \frac{E_{2 m}}{\left( 2 m \right) !} = - \frac{B_{2 m + 1,
  \chi_{- 4}}}{\left( 2 m + 1 \right) !} . \label{eq:EB}
\end{equation}
Generalized Bernoulli numbers are intimatly related to values of Dirichlet
$L$-series. Let $n > 0$ be an integer and $\chi$ a primitive Dirichlet
character of conductor $L$ such that $\chi (- 1) = (- 1)^n$. Then, as
detailed, for instance, in {\cite[Thm. 3.3.4]{miyake}},
\begin{eqnarray}
  L (n, \chi) & = & (- 1)^{n - 1} \frac{G (\chi)}{2} \left( \frac{2 \pi i}{L}
  \right)^n \frac{B_{n, \bar{\chi}}}{n!},  \label{eq:LB}\\
  L (1 - n, \chi) & = & - B_{n, \chi} / n. \nonumber
\end{eqnarray}
On the other hand, if $\chi (- 1) \neq (- 1)^n$, then $L (1 - n, \chi) = 0$
unless $\chi = 1$ and $n = 1$.

Finally, let us recall the basic transformation properties of the Eisenstein
series $E_k (\tau ; \chi, \psi)$, which are detailed, for instance, in
{\cite[Chapter 7]{miyake}}. Denote with $\Gamma_0 (L, M)$ the group of
matrices $\gamma = \left(\begin{array}{cc}
  a & b\\
  c & d
\end{array}\right) \in \tmop{SL}_2 (\mathbbm{Z})$ such that $M|b$ and $L|c$.
For any such $\gamma$,
\begin{equation}
  E_k (\tau ; \chi, \psi) |_k \gamma = \chi (d) \bar{\psi} (d) E_k (\tau ;
  \chi, \psi) . \label{eq:Eg}
\end{equation}
Moreover,
\begin{equation}
  E_k (\tau ; \chi, \psi) |_k S = \chi (- 1) E_k (\tau ; \psi, \chi) .
  \label{eq:ES}
\end{equation}
\section{Review of Eichler integrals}\label{sec:eichler}

In the language of Eichler integrals, to be reviewed next, the observation of
Section \ref{sec:eisenstein} becomes the simple statement that the secant
Dirichlet series $\psi_{2 m}$ are Eichler integrals of weight $2 m + 1$
Eisenstein series of level $2$. As such, it is natural that the $\psi_{2 m}$
satisfy modular functional equations as in Section \ref{sec:funres}. In fact,
it becomes \tmtextit{a priori} clear that the $\psi_{2 m}$ satisfy modular
relations such as (\ref{eq:phifunsym}), in which the coefficients of the
rational function on the right-hand side are determined by the period
polynomial of the Eichler integral.

Lemma \ref{lem:phiD} shows that
\[ F_{2 m} (\tau) \assign \psi_{2 m} (\tau) + \frac{E_{2 m}}{2 (2 m) !}  (\pi
   i \tau)^{2 m} \]
has the property that its $2 m$th derivative is a modular form with respect to
$\Gamma = \langle A, B \rangle$ of weight $2 m + 1$ (with $A$ and $B$ as
defined in (\ref{eq:AB})). In other words, $F_{2 m}$ is an {\tmem{Eichler
integral}} (we adopt the common custom and also refer to $\psi_{2 m}$ as an
Eichler integral). In particular, for all $\gamma = \left(\begin{array}{cc}
  a & b\\
  c & d
\end{array}\right) \in \Gamma$,
\begin{equation}
  (c \tau + d)^{2 m - 1} F_{2 m} (\gamma \tau) - F_{2 m} (\tau)
  \label{eq:periodpoly}
\end{equation}
is a polynomial of degree $2 m - 1$, the {\tmem{period polynomial}} of $F_{2
m}$. Rewriting (\ref{eq:periodpoly}) in terms of $\psi_{2 m} (\tau)$, we find
that $\psi_{2 m} (\tau)$ satisfies a functional equation of the form
(\ref{eq:phifunsym}), where the rational function on the right-hand side is
expressed in terms of the period polynomial of $F_{2 m}$. For the general
theory of period polynomials we refer to {\cite{popa-period}} and the
references therein. A very brief introduction, suitable for our purposes, is
given next.

\begin{remark}
  A direct way to see that (\ref{eq:periodpoly}) is indeed a polynomial of
  degree $2 m - 1$ is offered by Bol's identity {\cite{bol-de}}. It states
  that, for all sufficiently differentiable $F$ and $\gamma =
  \left(\begin{array}{cc}
    a & b\\
    c & d
  \end{array}\right) \in \tmop{SL}_2 (\mathbbm{R})$,
  \begin{equation}
    (D^{k + 1} F) (\gamma \tau) = (c \tau + d)^{k + 2} D^{k + 1} \left[ (c
    \tau + d)^k F (\gamma \tau) \right] . \label{eq:bol}
  \end{equation}
  For the present purpose, we apply (\ref{eq:bol}) with $F = F_{2 m}$ and $k =
  2 m - 1$.
\end{remark}

In the following we adopt the notation of {\cite{popa-period}}. Let
$\mathcal{A}$ be the space of holomorphic functions on the upper half-plane
$\mathcal{H}$. As usual, $\Gamma_1 = \tmop{SL}_2 (\mathbbm{Z})$ acts on
$\mathcal{H}$ by linear fractional transformations, and on $\mathcal{A}$ via
the slash operators; namely, if $f \in \mathcal{A}$ and $k$ is an integer,
then
\[ (f|_k g) (\tau) = (c \tau + d)^{- k} f (g \tau), \hspace{1em} g =
   \left(\begin{array}{cc}
     a & b\\
     c & d
   \end{array}\right) \in \Gamma_1 . \]
This action extends naturally to the group algebra $\mathbbm{C} [\Gamma_1]$.

As usual, we denote with $T$, $S$ and $R$ the matrices
\begin{equation}
  T = \left(\begin{array}{cc}
    1 & 1\\
    0 & 1
  \end{array}\right), \hspace{1em} S = \left(\begin{array}{cc}
    0 & - 1\\
    1 & 0
  \end{array}\right), \hspace{1em} R = \left(\begin{array}{cc}
    1 & 0\\
    1 & 1
  \end{array}\right), \label{eq:TS}
\end{equation}
and recall that the matrices $T$ and $S$ generate $\Gamma_1$.

From now on, let $f$ be a (not necessarily cuspidal) modular form of integral
weight $k \geqslant 2$ with respect to $\Gamma$, with $\Gamma$ being a
subgroup of finite index of $\Gamma_1$. Note that the modularity of $f$
implies that $f|_k g = f$ for all $g \in \Gamma$. In the sequel, we will
abbreviate $f|g = f|_k g$ since this is the only action of $\Gamma_1$ on
modular forms of weight $k$ that we consider.

Throughout this section, $w = k - 2$. Let $V_w$ be the space of complex
polynomials of degree at most $w$. The (multiple) period polynomial,
introduced in {\cite{popa-period}}, attached to $f$ is the map $\rho_f :
\Gamma \backslash \Gamma_1 \rightarrow V_w$ defined by
\begin{equation}
  \rho_f (A) (X) = \int_0^{i \infty} \left[ f|A (t) - a_0 (f|A) \right] (t -
  X)^w \mathd t. \label{eq:rhof}
\end{equation}
In the sequel, we will often omit the dependence on $X$ and just write $\rho_f
(A)$ for the left-hand side. The goal of the final set of definitions is to
connect these period polynomials, whose coefficients encode the critical
$L$-values of $f$, to the transformation properties of Eichler integrals of
$f$. The (multiple) Eichler integral of $f$, introduced in
{\cite{popa-period}}, is the function $\tilde{f} : \Gamma \backslash \Gamma_1
\rightarrow \mathcal{A}$ defined by
\begin{equation}
  \tilde{f} (A) (\tau) = \int_{\tau}^{i \infty} \left[ f|A (z) - a_0 (f|A)
  \right] (z - \tau)^w \mathd z, \label{eq:ftilda}
\end{equation}
with $a_0 (f) = f (i \infty)$ denoting the constant term of the Fourier
expansion of $f$. If $g \in \Gamma_1$ then $\tilde{f} |g (A) = \tilde{f} (A
g^{- 1}) |_{- w} g$ defines an action of $\Gamma_1$, and hence $\mathbbm{C}
[\Gamma_1]$, on functions $\tilde{f} : \Gamma \backslash \Gamma_1 \rightarrow
\mathcal{A}$.

The following result is {\cite[Proposition 8.1]{popa-period}}, which may also
be found in {\cite{weil-hecke}}, where it is expressed in slightly different
terms.

\begin{proposition}
  \label{prop:rhohatrho}With $f$ as above, define $\hat{\rho}_f = \tilde{f} |
  (1 - S)$. Then, for any $A \in \Gamma \backslash \Gamma_1$,
  \[ \hat{\rho}_f (A) = \rho_f (A) + (- 1)^w \frac{a_0 (f|A)}{w + 1} X^{w + 1}
     + \frac{a_0 (f|A S^{- 1})}{w + 1} X^{- 1} . \]
\end{proposition}

Note that matters simplify when only cusp forms are considered; in that case,
$\rho_f$ and $\hat{\rho}_f$ coincide. Though not necessarily a polynomial in
the Eisenstein case, we will also refer to $\hat{\rho}_f$, as well as to
$\tilde{f} | (1 - \gamma)$ for $\gamma \in \Gamma_1$, as period polynomials of
$f$.

\begin{example}
  Let us make these definitions very explicit in the case where $f$ is the
  Eisenstein series $E (\tau) = E_k (\tau ; \chi, \psi)$ introduced in
  (\ref{eq:eisenstein}). Then $E|S = \chi (- 1) E_k (\tau ; \psi, \chi)$ by
  (\ref{eq:ES}), and thus
  \begin{equation}
    \hat{\rho}_E (I) = \tilde{E}_k (X ; \chi, \psi) - \psi (- 1) X^{k - 2}
    \tilde{E}_k (- 1 / X ; \psi, \chi), \label{eq:rhohatEtilda}
  \end{equation}
  assuming that $\chi (- 1) \psi (- 1) = (- 1)^k$ (since, otherwise, $E = 0$).
  Here, and in the sequel, we denote, with a slight abuse of notation,
  \[ \tilde{E}_k (\tau ; \chi, \psi) = \tilde{E} (I) (\tau), \]
  where the right-hand side is defined by (\ref{eq:ftilda}). It follows from
  the Fourier expansion (\ref{eq:Eqexp}) that, for primitive $\psi$,
  \begin{equation}
    \tilde{E}_k (\tau ; \chi, \psi) = - \frac{4 \pi i}{k - 1}  \frac{G
    (\psi)}{M} \sum_{n = 1}^{\infty} \left( \sum_{d|n} \bar{\psi} (n / d) \chi
    (d) d^{1 - k} \right) e^{2 \pi i n \tau / M} . \label{eq:Etildaq}
  \end{equation}
\end{example}

\begin{example}
  \label{eg:phiEtilda}Following Example \ref{eg:phiDE}, we observe from
  (\ref{eq:phiDE}) that
  \begin{equation}
    \psi_{2 m} (\tau / 2) = \frac{2 m}{\pi}  \tilde{E}_{2 m + 1} (\tau ; 1,
    \chi_{- 4}), \label{eq:phiEtilda}
  \end{equation}
  thus making explicit the nature of $\psi_{2 m}$ as an Eichler integral. In
  the present context, the functional equation (\ref{eq:phifun}), on replacing
  $\tau$ with $\tau / 2$ and the appropriate scaling, translates into
  \begin{equation}
    \tilde{E}_{2 m + 1} (\tau ; 1, \chi_{- 4}) |_{1 - 2 m} (R - 1) =
    \frac{\pi^{2 m + 1}}{2 m}  \left[ z^{2 m - 1} \right]  \frac{\sin (\tau z
    / 2)}{\sin (z) \sin ((\tau + 1) z)}, \label{eq:phifunE}
  \end{equation}
  where $R$ is as in (\ref{eq:TS}).
\end{example}

The next result allows us to express the left-hand side of $(
\ref{eq:phifunE})$ in terms of $\hat{\rho}_{E_{2 m + 1}}$ which, in the sense
of Proposition \ref{prop:rhohatrho}, is a period polynomial of $E_{2 m + 1}$.

\begin{proposition}
  \label{prop:ftildarhohat}Let $R$ be as in (\ref{eq:TS}). Let $f$ be a
  modular form for a group $\Gamma \leqslant \Gamma_1$, and let $n$ be such
  that $R^n \in \Gamma$. Then
  \[ \tilde{f} | (1 - R^n) (I) = \hat{\rho}_f | (1 - R^n) (I) . \]
\end{proposition}

\begin{proof}
  For any $C = \left(\begin{array}{cc}
    a & b\\
    c & d
  \end{array}\right) \in \Gamma$, it follows from the definition of
  $\hat{\rho}_f$ that
  \[ \hat{\rho}_f | (1 - C) = \tilde{f} | (1 - S) (1 - C) = \tilde{f} | (1 -
     C) - \tilde{f} | (S - S C) . \]
  It therefore suffices to show that $\tilde{f} |S (I) = \tilde{f} |S R^n
  (I)$. To see this, we observe that, for $C \in \Gamma$ as above,
  \[ \tilde{f} |S C (I) = \tilde{f} (S^{- 1}) |_{- w} S C = (a \tau + b)^w
     \int_{S C \tau}^{i \infty} \left[ f|S^{- 1} (z) - a_0 (f|S^{- 1}) \right]
     (z - S C \tau)^w \mathd z. \]
  The change of variables $z = S C S^{- 1} z'$ yields
  \[ \tilde{f} |S C (I) = \tau^w \int_{S \tau}^{a / b} \left[ f|S^{- 1} (z) -
     (- b z + a)^{- k} a_0 (f|S^{- 1}) \right] (z - S \tau)^w \mathd z, \]
  and the desired equality follows because, upon setting $a = 1$ and $b = 0$,
  the right-hand side does not depend on the value of $c$.
\end{proof}

\section{Period polynomials of Eisenstein series}\label{sec:periods}

In the case $\chi = 1$ and $\psi = 1$, the Eisenstein series $E (\tau) = E_{2
k} (\tau ; 1, 1)$ is the usual Eisenstein series of weight $2 k$ with respect
to the full modular group. Its period polynomial, defined in Proposition
\ref{prop:rhohatrho} and made explicit in (\ref{eq:rhohatEtilda}) for
Eisenstein series, is well-known to be
\begin{equation}
  \hat{\rho}_E (I) = - \frac{(2 \pi i)^{2 k}}{2 k - 1}  \left[ \sum_{s = 0}^k
  \frac{B_{2 s}}{(2 s) !} \frac{B_{2 k - 2 s}}{\left( 2 k - 2 s \right) !}
  X^{2 k - 2 s - 1} + \frac{\zeta (2 k - 1)}{(2 \pi i)^{2 k - 1}} (X^{2 k - 2}
  - 1) \right] ; \label{eq:rhohat1}
\end{equation}
compare, for instance, {\cite[(11)]{zagier-pp91}}. On the other hand, this
evaluation is equivalent to the formula (\ref{eq:rama}), which Ramanujan
famously recorded and which we briefly discuss in Section \ref{sec:rama1}. A
beautiful account of this connection is contained in
{\cite{gun-murty-rath2011}}.

The (Laurent) polynomials on the right-hand side of (\ref{eq:rhohat1}) have
interesting properties, which have been studied, for instance, in
{\cite{gun-murty-rath2011}}, {\cite{msw-zeroes}}, and {\cite{lr-unimod13}}. In
Section \ref{sec:rpx}, we indicate that the generalized polynomials obtained
in this section for higher level share similar properties.

In light of Example \ref{eg:phiEtilda} and Proposition
\ref{prop:ftildarhohat}, the functional equations satisfied by the secant
Dirichlet series are determined by the period polynomials associated to the
Eisenstein series $E_{2 k + 1} (\tau ; 1, \chi_{- 4})$. We next compute the
period polynomials of the Eisenstein series $E_{2 k + 1} (\tau ; \chi, \psi)$,
with $\chi$ and $\psi$ being any pair of primitive Dirichlet characters.

\begin{theorem}
  \label{thm:ppE}Let $k \geqslant 3$, and let $\chi$ and $\psi$ be primitive
  Dirichlet characters modulo $L$ and $M$, respectively, such that $\chi (- 1)
  \psi (- 1) = (- 1)^k$. For the Eisenstein series $E (\tau) = E_k (\tau ;
  \chi, \psi)$, defined in (\ref{eq:eisenstein}),
  \begin{eqnarray}
    \hat{\rho}_E (I) & = & - \frac{4 \psi (- 1)}{k - 1}
    \sum_{\substack{s=0 \\ \chi(-1)=(-1)^s}}^k
    L (s, \chi) L (k - s, \psi) X^{k - s - 1} \nonumber\\
    &  & - \frac{2 \psi (- 1)}{k - 1} \pi i \left[ \varepsilon_{\chi} L (k -
    1, \psi) X^{k - 2} - \varepsilon_{\psi} L (k - 1, \chi) \right] . 
    \label{eq:rhohatE}
  \end{eqnarray}
  Here, $\varepsilon_{\chi} = 1$ if $\chi = 1$, and $\varepsilon_{\chi} = 0$
  otherwise.
\end{theorem}

\begin{proof}
  We first observe, from the definition (\ref{eq:rhof}), the general fact
  that, for a modular form $E$ of weight $k$,
  \[ \rho_E (A) = (- 1)^{k - 1} \sum_{s = 1}^{k - 1} \binom{k - 2}{s - 1}
     \frac{\Gamma (s)}{(2 \pi i)^s} L (E|A, s) X^{k - s - 1}, \]
  with the $L$-function of $E$ as defined in (\ref{eq:Lf}). Let $s$ be an
  integer with $0 < s < k$. We deduce from (\ref{eq:LE}) and the functional
  equation, given, for instance, in {\cite[Theorem 12.11, p. 263]{apostol1}},
  \begin{equation}
    L (1 - s, \bar{\psi}) = \frac{M^{s - 1} \Gamma (s)}{(2 \pi)^s} (e^{- \pi i
    s / 2} + \psi (- 1) e^{\pi i s / 2}) \tau ( \bar{\psi}) L (s, \psi)
    \label{eq:LpsiFE}
  \end{equation}
  that, for $s$ such that $\chi (- 1) = (- 1)^s$,
  \[ L (E, s) = 4 (- 1)^s  (2 \pi i)^s  \frac{\Gamma (k - s)}{\Gamma (k)} L
     (s, \chi) L (k - s, \psi) . \]
  On the other hand, for $s$ such that $\chi \left( - 1 \right) \neq \left( -
  1 \right)^s$, we have $L ( \bar{\psi}, 1 - k + s) = 0$, which implies that
  $L \left( E, s \right) = 0$ unless $s = 1$, $\chi = 1$ or $s = k - 1$, $\psi
  = 1$. Combining these, we find that
  \begin{eqnarray*}
    \rho_E (I) & = & - \frac{4 \psi (- 1)}{k - 1}
    \sum_{\substack{s=1 \\ \chi(-1)=(-1)^s}}^{k-1}
    L (s, \chi) L (k - s, \psi) X^{k - s - 1}\\
    &  & - \frac{(- 1)^k}{2 \pi i} \varepsilon_{\chi} L (E, 1) X^{k - 2} -
    \frac{(- 1)^k}{(2 \pi i)^{k - 1}} \Gamma (k - 1) \varepsilon_{\psi} L (E,
    k - 1),
  \end{eqnarray*}
  where the sum is over all integers $s$ such that $0 < s < k$ and $\chi (- 1)
  = (- 1)^s$. From (\ref{eq:LE}), together with the functional equation
  (\ref{eq:LpsiFE}) of the involved Dirichlet $L$-series, we deduce that,
  assuming $\chi = 1$,
  \[ L (E, 1) = \frac{(2 \pi i)^2}{k - 1} L (k - 1, \psi) . \]
  On the other hand, if $\psi = 1$, then
  \[ L (E, k - 1) = - \frac{(- 2 \pi i)^k}{(k - 1) !} L (k - 1, \chi) . \]
  It follows from Proposition \ref{prop:rhohatrho} that
  \[ \hat{\rho}_E (I) = \rho_E (I) + (- 1)^k  \frac{a_0 (E)}{k - 1} X^{k - 1}
     + \frac{a_0 (E|S^{- 1})}{k - 1} X^{- 1} . \]
  The values for $a_0 (E)$ and $a_0 (E|S^{- 1}) = (- 1)^k a_0 (E|S)$ are given
  by (\ref{eq:Eqexp}) in combination with (\ref{eq:ES}). Finally, using the
  fact that $L (0, \chi) = 0$ for any even Dirichlet character $\chi \neq 1$,
  we obtain (\ref{eq:rhohatE}).
\end{proof}

Observe that, in the case $\chi = \psi = 1$, using Euler's identity
{\cite{ayoub1}}
\[ \zeta (2 m) = (- 1)^{m + 1} \frac{B_{2 m}}{2 (2 m) !}  (2 \pi)^{2 m}, \]
we obtain from Theorem \ref{thm:ppE} the well-known special case
(\ref{eq:rhohat1}). In the same spirit, Theorem \ref{thm:ppE} may always be
rewritten, using (\ref{eq:LB}), in terms of generalized Bernoulli numbers as
we record below. For the cases where $\chi = 1$ or $\psi = 1$ the appropriate
extra terms need to be inserted.

\begin{corollary}
  \label{cor:rhohat}Under the assumptions of Theorem \ref{thm:ppE}, if,
  additionally, $\chi$ and $\psi$ are both nonprincipal, then
  \[ \hat{\rho}_E (I) = - \chi (- 1) G (\chi) G (\psi) \frac{(2 \pi i)^k}{k -
     1} \sum_{s = 0}^k \frac{B_{k - s, \bar{\chi}}}{(k - s) !L^{k - s}} 
     \frac{B_{s, \bar{\psi}}}{s!M^s} X^{s - 1} . \]
\end{corollary}

\begin{remark}
  We note that similar results, based on residue calculations in the spirit of
  Section \ref{sec:funres}, are obtained in {\cite{berndt-eis75}}. In fact,
  the Eisenstein series considered in {\cite{berndt-eis75}}, namely
  \[ \sumprime_{m, n \in \mathbbm{Z}} \frac{\chi (m) \psi (n)}{((m + r_1) \tau + n
     + r_2)^k}, \]
  have the two extra parameters $r_1$ and $r_2$ in comparison with $E_k (\tau
  ; \chi, \psi)$. However, the analysis in {\cite{berndt-eis75}} is restricted
  to the case when $\chi$ and $\psi$ are primitive characters of the same
  modulus.
\end{remark}

\begin{example}
  \label{eg:katayama}In the case $\chi = \psi = 1$, Theorem \ref{thm:ppE}
  reduces to Ramanujan's identity (\ref{eq:rhohat1}), which in particular
  yields an interesting formula for the odd zeta values $\zeta (2 k - 1)$; see
  also {\cite{grosswald-zeta70}}. This formula has inspired much research,
  such as {\cite{gun-murty-rath2011}}. Analogous formulas for other Dirichlet
  $L$-series have been derived in {\cite{katayama-L74}} using partial fraction
  expansions. We will now illustrate how the results of {\cite{katayama-L74}}
  follow from the special case $\psi = 1$ of Theorem \ref{thm:ppE}, thus
  providing an alternative proof. In the setting of Theorem \ref{thm:ppE},
  with $\psi = 1$ and $\chi \neq 1$, we have
  \[ \hat{\rho}_E (I) = \frac{2 \pi i}{k - 1} L (k - 1, \chi) - \frac{4}{k -
     1} \sum_{j = 0}^{\lfloor k / 2 \rfloor} L (k - 2 j, \chi) \zeta (2 j)
     X^{2 j - 1} . \]
  On the other hand, from (\ref{eq:rhohatEtilda}),
  \[ \hat{\rho}_E (I) = \tilde{E}_k (\tau ; \chi, 1) - \tau^{k - 2}
     \tilde{E}_k (- 1 / \tau ; 1, \chi) . \]
  Using the Fourier expansion (\ref{eq:Etildaq}), as well as $G (\chi) G (
  \bar{\chi}) = \chi (- 1) L$ and the simple summations
  \begin{eqnarray*}
    \sum_{n = 1}^{\infty} \left( \sum_{d|n} \chi (d) d^{1 - k} \right) q^n & =
    & \sum_{n = 1}^{\infty} \frac{\chi (n)}{n^{k - 1}}  \frac{q^n}{1 - q^n},\\
    \sum_{n = 1}^{\infty} \left( \sum_{d|n} \chi (n / d) d^{1 - k} \right) q^n
    & = & \sum_{a = 1}^L \chi (a) \sum_{n = 1}^{\infty} \frac{1}{n^{k - 1}} 
    \frac{q^{a n}}{1 - q^{L n}},
  \end{eqnarray*}
  we obtain
  \[ \hat{\rho}_E (I) = \frac{4 \pi i}{k - 1} \left[ F_1 (\tau) - \frac{(-
     \tau)^{k - 2}}{G ( \bar{\chi})} F_2 (- 1 / \tau) \right], \]
  where, similar to {\cite{katayama-L74}},
  \begin{eqnarray*}
    F_1 (\tau) & = & \sum_{n = 1}^{\infty} \frac{\chi (n)}{n^{k - 1}} 
    \frac{e^{2 \pi i n \tau}}{e^{2 \pi i n \tau} - 1},\\
    F_2 (\tau) & = & \sum_{a = 1}^L \bar{\chi} (a) \sum_{n = 1}^{\infty}
    \frac{1}{n^{k - 1}}  \frac{e^{2 \pi i a n \tau / L}}{e^{2 \pi i n \tau} -
    1} .
  \end{eqnarray*}
  Solving for $L (k - 1, \chi)$, we have arrived at
  \[ \frac{1}{2} L (k - 1, \chi) = F_1 (\tau) - \frac{(- \tau)^{k - 2}}{G (
     \bar{\chi})} F_2 (- 1 / \tau) + \frac{1}{\pi i} \sum_{j = 0}^{\lfloor k /
     2 \rfloor} L (k - 2 j, \chi) \zeta (2 j) \tau^{2 j - 1}, \]
  which is equivalent to the main result of {\cite{katayama-L74}}. Note that
  this formula expresses the $L$-value as a combination of two Eichler
  integrals and a power of $\pi$ times a Laurent polynomial with rational
  coefficients. In Example \ref{eg:LasEichler}, similar formulas for these
  $L$-values are given, where only one Eichler integral is involved (evaluated
  at two arguments) and the polynomials appear to have the additional property
  of having all their nonreal roots on the unit circle.
\end{example}

With Theorem \ref{thm:ppE} in place, it is easy to deduce functional equations
such as (\ref{eq:phifun}), which is the case $\chi = 1$, $\psi = \chi_{- 4}$
of the next result. Note that the restriction to $\psi \neq 1$ is just to
avoid the presence of an additional term. The case $\psi = 1$ is discussed in
Example \ref{eg:LasEichler}, where the formulas promised at the end of Example
\ref{eg:katayama} are derived.

\begin{corollary}
  \label{cor:EtildaB}Let $k \geqslant 3$, and let $\chi$ and $\psi$ be
  primitive Dirichlet characters modulo $L$ and $M$, respectively, such that
  $\chi (- 1) \psi (- 1) = (- 1)^k$. Let $R$ be as in (\ref{eq:TS}). If $\psi
  \neq 1$, then, for any integer $n$ such that $L|n$,
  \begin{equation*}
    \tilde{E}_k (X ; \chi, \psi) |_{2 - k} (1 - R^n)
    = - \frac{4 \psi (- 1)}{k - 1} \sum_{\substack{s=0 \\ \chi(-1)=(-1)^s}}^k
    L (s, \chi) L (k - s, \psi) X^{k - s - 1} (1 - (n X +
    1)^{s - 1}) .
  \end{equation*}
\end{corollary}

\begin{proof}
  Recall from (\ref{eq:Eg}) that $E (\tau) = E_k (\tau ; \chi, \psi)$ is
  modular with respect to $\Gamma_0 (L, M)$. Since $R^n \in \Gamma_0 (L, M)$,
  we may apply Proposition \ref{prop:ftildarhohat} and Theorem \ref{thm:ppE}
  to obtain
  \[ \tilde{E}_k (X ; \chi, \psi) |_{2 - k} (1 - R^n) = \tilde{E} | (1 - R^n)
     (I) = \hat{\rho}_E (I) |_{- w} (1 - R^n), \]
  with the Laurent polynomial $\hat{\rho}_E (I)$ given explicitly in
  (\ref{eq:rhohatE}). Note that
  \[ X^j |_{- w} \left( 1 - R^n \right) = X^j - \left( n X + 1 \right)^w
     \left( \frac{X}{n X + 1} \right)^j = X^j \left( 1 - \left( n X + 1
     \right)^{w - j} \right) . \]
  In particular, $X^w |_{- w} B = X^w$, so that the term in (\ref{eq:rhohatE})
  involving $L (k - 1, \psi)$, if at all existent, is eliminated in $\tilde{E}
  | (1 - R^n) (I)$.
\end{proof}

\begin{example}
  As indicated in Example \ref{eg:phiEtilda}, Corollary \ref{cor:EtildaB}
  specializes to a variation of Theorem \ref{thm:fun} on setting $\chi = 1$
  and $\psi = \chi_{- 4}$. Namely, using (\ref{eq:EB}) to relate the
  generalized Bernoulli numbers to Bernoulli and Euler numbers, we find that,
  for any positive integer $m$, the secant Dirichlet series satisfies the
  functional equation
  \begin{eqnarray}
    &  & (2 \tau + 1)^{2 m - 1} \psi_{2 m} \left( \frac{\tau}{2 \tau + 1}
    \right) - \psi_{2 m} (\tau) \nonumber\\
    & = & (\pi i)^{2 m} \sum_{n = 0}^m \frac{2^{2 n - 1} B_{2 n} E_{2 m - 2
    n}}{(2 n) ! (2 m - 2 n) !} \tau^{2 m - 2 n} \left[ 1 - (2 \tau + 1)^{2 n -
    1} \right] .  \label{eq:phifun2}
  \end{eqnarray}
  We note that (\ref{eq:phifun2}) is the same functional equation as
  (\ref{eq:phifun}) but representing the right-hand side in a somewhat
  different way; see Remark \ref{rk:phifun-12}.
\end{example}

\begin{remark}
  \label{rk:phifun-12}Let us indicate how to see, in a direct fashion, that
  (\ref{eq:phifun}) and (\ref{eq:phifun2}) represent the same functional
  equation. Note that (\ref{eq:phifun2}) can be expressed as
  \begin{equation}
    \psi_{2 m} (\tau) |_{1 - 2 m} (1 - B) = \tau^{2 m - 1} \left[ h_{2 m}
    \left( \tfrac{1}{\tau} \right) - h_{2 m} \left( 2 + \tfrac{1}{\tau}
    \right) \right], \label{eq:phifun2h}
  \end{equation}
  where $h_{2 m} (\tau)$ is the rational function
  \begin{equation}
    h_{2 m} (\tau) = (\pi i)^{2 m} \sum_{n = 0}^m \frac{B_{2 n} E_{2 m - 2
    n}}{(2 n) ! (2 m - 2 n) !}  (2 \tau)^{2 n - 1} . \label{eq:def-h}
  \end{equation}
  By the definitions of the Bernoulli and Euler numbers, (\ref{eq:B}) and
  (\ref{eq:defen}), we find that
  \[ h_{2 m} (\tau) = \tfrac{1}{2} \pi^{2 m} [z^{2 m - 1}] \cot (\tau z) \sec
     (z) . \]
  Hence, the right-hand side of (\ref{eq:phifun2}) and (\ref{eq:phifun2h})
  equals
  \[ \tfrac{1}{2} \pi^{2 m} [z^{2 m - 1}] \sec (\tau z) \left[ \cot (z) - \cot
     ((2 \tau + 1) z) \right] . \]
  The equivalence of (\ref{eq:phifun}) and (\ref{eq:phifun2}) then follows
  from
  \[ \sec (\tau z) \left[ \cot (z) - \cot ((2 \tau + 1) z) \right] =
     \frac{\sin (\tau z)}{\sin (z) \sin ((2 \tau + 1) z)}, \]
  which is obtained from basic trigonometric identities.
\end{remark}

As an application of Corollary \ref{cor:EtildaB}, we now derive formulas for
the values of $\psi_{2 m}$ at families of real quadratic irrationalities, thus
complementing and illustrating the results of Section \ref{sec:sqrt} in an
explicit fashion.

\begin{example}
  \label{eg:phisqrtvals}Since Corollary \ref{cor:EtildaB} applies to all
  powers of $R$, which lie in the appropriate modular subgroup, we find that,
  for positive integers $m$ and integers $\mu$,
  \begin{eqnarray}
    &  & (2 \mu \tau + 1)^{2 m - 1} \psi_{2 m} \left( \frac{\tau}{2 \mu \tau
    + 1} \right) - \psi_{2 m} (\tau) \nonumber\\
    & = & (\pi i)^{2 m} \sum_{n = 0}^m \frac{2^{2 n - 1} B_{2 n} E_{2 m - 2
    n}}{(2 n) ! (2 m - 2 n) !} \tau^{2 m - 2 n} \left[ 1 - (2 \mu \tau + 1)^{2
    n - 1} \right] .  \label{eq:phifun2x}
  \end{eqnarray}
  We now demonstrate how to use these functional equations to obtain families
  of explicit evaluations of $\psi_{2 m}$ at certain real quadratic
  irrationalities. Let $\tau_0$ be fixed by $A^{\lambda} B^{\mu} A^{\nu}$,
  that is $A^{\lambda} B^{\mu} A^{\nu} \tau_0 = \tau_0$. A brief calculation
  shows that
  \[ \tau_0 = \lambda - \nu \pm \sqrt{(\lambda + \nu) \left( \tfrac{1}{\mu} +
     (\lambda + \nu) \right)} . \]
  Denote with $T_{m, \mu} = \psi_{2 m} |_{1 - 2 m} (B^{\mu} - I)$ the
  right-hand side of (\ref{eq:phifun2x}). It follows from
  \[ \psi_{2 m} (A^{\lambda} B^{\mu} A^{\nu} \tau) = \frac{\psi_{2 m} (A^{\nu}
     \tau) + T_{m, \mu} (A^{\nu} \tau)}{(2 \mu A^{\nu} \tau + 1)^{2 m - 1}},
  \]
  together with the fact that $\tau_0$ is fixed by $A^{\lambda} B^{\mu}
  A^{\nu}$, that
  \[ \psi_{2 m} (\tau_0) = \frac{T_{m, \mu} (\tau_0 + 2 \nu)}{(2 \mu (\tau_0 +
     2 \nu) + 1)^{2 m - 1} - 1} . \]
  A straightforward, but slightly tedious, calculation using
  (\ref{eq:phifun2x}) and the explicit value of $\tau_0$ shows that
  \begin{equation}
    \psi_{2 m} (\tau_0) = - (\pi i)^{2 m} \sum_{n = 0}^m \frac{2^{2 n -
    1}}{\mu^{2 m - 2 n}}  \frac{B_{2 n} E_{2 m - 2 n}}{(2 n) ! (2 m - 2 n) !} 
    \frac{r_n}{r_m}, \label{eq:phitau0}
  \end{equation}
  where $r_n$ are the rational numbers
  \[ r_n = \frac{1}{2} \left[ (1 + \sqrt{\alpha})^{2 n - 1} + (1 -
     \sqrt{\alpha})^{2 n - 1} \right] = \left\{ \begin{array}{ll}
       \frac{1}{1 - \alpha}, & \text{if $n = 0$},\\
       \sum_{j = 0}^{n - 1} \binom{2 n - 1}{2 j} \alpha^j, & \text{if $n
       \geqslant 1$},
     \end{array} \right. \]
  and $\alpha = \frac{1}{\mu (\nu + \lambda)} + 1$. Note that this is an
  explicit illustration of the general fact, proved in Theorem
  \ref{thm:phisqrtr}, that $\psi_{2 m} (\tau)$ is a rational multiple of
  $\pi^{2 m}$ whenever $\tau$ is a real quadratic irrationality.
  
  We note that the right-hand side of (\ref{eq:phitau0}) only depends on
  $\kappa = \nu + \lambda$ but not on $\nu$ and $\lambda$ individually. For
  the left-hand side, this follows from the obvious periodicity relation
  $\psi_{2 m} (\tau + 2) = \psi_{2 m} (\tau)$. The first two cases of
  (\ref{eq:phitau0}) can thus be stated, in equivalent forms, as
  \begin{eqnarray*}
    \psi_2 \left( \kappa + \sqrt{\kappa \left( \tfrac{1}{\mu} + \kappa
    \right)} \right) & = & \frac{\pi^2}{6}  \left( 1 + \frac{3 \kappa}{2 \mu}
    \right),\\
    \psi_4 \left( \kappa + \sqrt{\kappa \left( \tfrac{1}{\mu} + \kappa
    \right)} \right) & = & \frac{\pi^4}{90} \left( 1 + \frac{5 \kappa}{2 \mu}
    - \frac{5 \kappa^2 (16 \mu^2 - 15)}{8 \mu^2 (4 \kappa \mu + 3)} \right),
  \end{eqnarray*}
  where $\kappa$ and $\mu$ are integers and $\mu \neq 0$. The first of these,
  in the special case $\kappa = 2 \lambda$, is also given in
  {\cite{lrr-secantzeta}}.
\end{example}

\begin{remark}
  Using the methods of {\cite{berndt-char73}}, {\cite{berndt-eis75}} or,
  alternatively, {\cite{razar-dir77}}, one can derive the general
  transformation laws of the Eichler integrals $\tilde{E}_k (\tau ; \chi,
  \psi)$, and, in particular, $\psi_{2 m} (\tau)$, under arbitrary elements of
  the full modular group. Here, we do not, however, pursue this further.
\end{remark}

\section{Zeros of generalized Ramanujan polynomials}\label{sec:rpx}

It has recently been shown in {\cite{cfi-ppz}} and {\cite{er-ppz13}} that the
nontrivial zeros of period polynomials of modular forms, which are Hecke
eigenforms of level $1$, all lie on the unit circle. In this section, we
consider the Eisenstein case of higher level by investigating the zeros of the
period polynomials calculated in the previous section. We again find that, at
least conjecturally, most of the roots lie on a circle in the complex plane.
The observations suggest that the problem solved by {\cite{cfi-ppz}} and
{\cite{er-ppz13}} is interesting in the higher level case as well.

An application of these considerations is that knowledge of the location of
the zeros of the period polynomials calculated in the previous section gives
rise to explicit formulas for Dirichlet $L$-values of ``wrong'' parity (that
is, values at integers of parity opposite to the Dirichlet character) in terms
of Eichler integrals. This is made explicit in Example \ref{eg:LasEichler}.
The special case of the principal character is detailed in
{\cite{gun-murty-rath2011}}, in which case odd zeta values are expressed in
terms of the difference of two Eichler integrals at algebraic argument of
modulus $1$.

For positive integer $k$, and Dirichlet characters $\chi$ and $\psi$ modulo
$L$ and $M$, we define the {\tmem{generalized Ramanujan polynomial}}
\begin{equation}
  R_k (X ; \chi, \psi) = \sum_{s = 0}^k \frac{B_{s, \chi}}{s!} \frac{B_{k - s,
  \psi}}{\left( k - s \right) !} \left( \frac{X - 1}{M} \right)^{k - s - 1} (1
  - X^{s - 1}) . \label{eq:rx}
\end{equation}
Note that this is a polynomial if $\chi$ and $\psi$ are both nonprincipal, and
a Laurent polynomial otherwise. Further note that, if $\chi (- 1) \psi (- 1)
\neq (- 1)^k$, then $R_k (X ; \chi, \psi) = 0$, unless $\psi = 1$ in which
case
\[ R_k (X ; \chi, \psi) = \frac{1}{2}  \frac{B_{k - 1, \chi}}{(k - 1) !} (1 -
   X^{k - 2}) . \]
In the sequel, we will therefore often assume, without loss of generality and
as we did in previous sections, that $\chi (- 1) \psi (- 1) = (- 1)^k$.

In {\cite{msw-zeroes}} the {\tmem{Ramanujan polynomials}} are, essentially,
defined as
\begin{equation}
  R_k (X) = \sum_{s = 0}^k \frac{B_s}{s!}  \frac{B_{k - s}}{(k - s) !} X^{s -
  1}, \label{eq:RkX}
\end{equation}
where $k$ is an even integer (in {\cite{msw-zeroes}} the index $k$ is shifted
by $1$, $X$ appears with exponent $s$, and the definition differs for $k =
2$). The next result shows that the generalized Ramanujan polynomials, despite
their different appearance, reduce to the Ramanujan polynomials when $\chi =
1$ and $\psi = 1$.

\begin{proposition}
  \label{prop:rrx}For $k > 1$, $R_{2 k} (X ; 1, 1) = R_{2 k} (X)$.
\end{proposition}

\begin{proof}
  As evidenced by (\ref{eq:rhohat1}), the polynomial $R_{2 k} (X)$ is the odd
  part of the period polynomial of the level $1$ Eisenstein series $E_{2 k}
  (\tau ; 1, 1)$. As such it satisfies the relations {\cite{zagier-pp91}}
  \begin{equation}
    R_{2 k} (X) |_{2 - 2 k} (1 + S) = R_{2 k} (X) |_{2 - 2 k} (1 + U + U^2) =
    0. \label{eq:r1rel}
  \end{equation}
  Here, $U = T S$ with $T$ and $S$ as defined in (\ref{eq:TS}). On the other
  hand, by construction (\ref{eq:rx}),
  \[ R_{2 k} (X ; 1, 1) = R_{2 k} (X) |_{2 - 2 k} (1 - R) T^{- 1} . \]
  A brief calculation reveals that $(1 + U + U^2) T = T + R + S$. Hence, using
  both relations (\ref{eq:r1rel}), we find
  \[ R_{2 k} (X) |_{2 - 2 k} (1 - R) = R_{2 k} (X) |_{2 - 2 k} (1 + T + S) =
     R_{2 k} (X) |_{2 - 2 k} T, \]
  which proves the claim.
\end{proof}

The next example indicates that the definition (\ref{eq:rx}) of the
generalized Ramanujan polynomials is natural, by connecting them to period
polynomials of generalized Eisenstein series studied in Section
\ref{sec:periods}.

\begin{example}
  \label{eg:rperiod}As in Corollary \ref{cor:EtildaB}, let $k \geqslant 3$,
  and let $\chi$ and $\psi \neq 1$ be primitive Dirichlet characters modulo
  $L$ and $M$, respectively, such that $\chi (- 1) \psi (- 1) = (- 1)^k$.
  Then, with $R$ as in (\ref{eq:TS}),
  \[ \tilde{E}_k (X ; \chi, \psi) |_{2 - k} (1 - R^L) = - \chi (- 1) G (\chi)
     G (\psi) \frac{(2 \pi i / L)^k}{k - 1}  \frac{L}{M} R_k (L X + 1 ;
     \bar{\chi}, \bar{\psi}) . \]
  In other words, up to some scaling and a linear change of variables, the
  polynomial $R_k (X ; \bar{\chi}, \bar{\psi})$ is a period polynomial of the
  Eisenstein series $E_k (\tau ; \chi, \psi)$.
\end{example}

\begin{conjecture}
  \label{conj:unimod}For nonprincipal real Dirichlet characters $\chi$ and
  $\psi$, the polynomial $R_k (X ; \chi, \psi)$ is unimodular, that is, all
  its roots lie on the unit circle.
\end{conjecture}

We have verified Conjecture \ref{conj:unimod} numerically for all $k \leqslant
50$ and all characters of modulus up to $100$. We note that it follows from
Proposition \ref{prop:rrx} and the results in {\cite{msw-zeroes}} that all
nonreal zeros of $R_k (X ; 1, 1)$ lie on the unit circle. On the other hand,
it is conjectured (in equivalent form) in {\cite{lrr-secantzeta}} that all
roots of $R_k (X ; 1, \chi_{- 4})$ lie on the unit circle. Before giving
further evidence in support of Conjecture \ref{conj:unimod} as well as an
application, we indicate the conjectural situation in the cases $\chi = 1$ or
$\psi = 1$, which is not included above. In vague summary, it appears that
still most of the roots lie on the unit circle.

\begin{example}
  Let $\chi$ be a nonprincipal real Dirichlet character. Computations show
  that, at least for $k \leqslant 50$ and $\chi$ of modulus at most $100$, the
  polynomials $R_k (X ; \chi, 1)$ are unimodular, except when $\chi$ takes the
  same values as $\chi_3$, the unique character of conductor $3$. In the case
  $\chi = 1$, it was shown in {\cite{msw-zeroes}} that $R_{2 k} (X ; 1, 1)$,
  for $k \geqslant 2$, has exactly four distinct real roots (approaching $\pm
  2^{\pm 1}$) and that the remaining roots lie on the unit circle. On the
  other hand, it appears that $R_{2 k + 1} (X ; \chi_3, 1)$, for $k \geqslant
  3$, has exactly three distinct real roots ($- 1$ as well as two roots
  approaching $- 2^{\pm 1}$) and that the remaining roots again lie on the
  unit circle.
\end{example}

In the next example, we restrict to primitive characters for expositional
reasons and make use of the fact {\cite{zr-dirichlet}} that, for a given
conductor $M$, there is at most one primitive real Dirichlet character modulo
$M$ of each parity. We label even such characters as $M \upl$ and odd ones as
$M \um$. For instance, the label $8 \upl$ refers to the even real Dirichlet
character of conductor $8$.

\begin{example}
  The situation in the case of $R_k (X ; 1, \psi)$, with $\psi$ a real
  primitive character, is slightly more varied. For certain characters $\psi$,
  such as
  \[ 3 \um, 4 \um, 5 \upl, 8 \upm, 11 \um, 12 \upl, 13 \upl, 19 \um, 21 \upl,
     24 \upl, \ldots, \]
  the polynomial $R_k (X ; 1, \psi)$ appears to again be unimodular. For
  certain other characters $\psi$, such as
  \[ 1 \upl, 7 \um, 15 \um, 17 \upl, 20 \um, 23 \um, 24 \um, \ldots, \]
  we observe, at least for small $k$, that all nonreal roots lie on the unit
  circle. On the other hand, there remains a third group of exceptional
  characters $\psi$, namely $35 \um, 59 \um, 83 \um, 131 \um, 155 \um, 179
  \um, \ldots$ (we observe that in each listed case $\psi$ is odd), for which
  $R_k (X ; 1, \psi)$ can have nonreal zeros off the unit circle. Consider,
  for instance, the unique real primitive Dirichlet character $\chi_{35}$ of
  conductor $35$. Then the polynomial $R_7 (X ; 1, \chi_{35})$ has the seven
  roots (given to three decimal digits)
  \[ 1, \hspace{1em} 0.461 \pm 0.888 i, \hspace{1em} (- 0.657 \pm 0.922
     i)^{\pm 1} . \]
  While the first three listed roots have absolute value $1$, the last four
  have absolute value $1.132^{\pm 1}$. In each of the exceptional cases, we
  observed, as in the example of $R_7 (X ; 1, \chi_{35})$, at most four
  nonreal zeros off the unit circle (in light of Proposition \ref{prop:rsi},
  such zeros necessarily come in groups of four).
\end{example}

In order for all zeros of a polynomial $p (X) = a_0 + a_1 X + \cdots + a_n
X^n$, $a_n \neq 0$, to lie on the unit circle, it is a necessary condition
{\cite{cohn-roots22}}, {\cite{ls-roots12}}, that the polynomial is
self-inversive, that is, for some $\varepsilon$ with $| \varepsilon | = 1$,
$a_k = \varepsilon \overline{a_{n - k}}$ for $k = 0, 1, \ldots, n$. In support
of Conjecture \ref{conj:unimod}, we now observe that, for real characters,
$R_k (X ; \chi, \psi)$ is self-inverse with $\varepsilon = \pm 1$. In other
words, $R_k (X ; \chi, \psi)$ is reciprocal or anti-reciprocal depending on
the parity of $\psi$.

\begin{proposition}
  \label{prop:rsi}Let $\chi$ and $\psi$ be real Dirichlet characters. If $\chi
  (- 1) \psi (- 1) = (- 1)^k$, then
  \[ R_k (X ; \chi, \psi) = \psi (- 1) X^{k - 2} R_k (X^{- 1} ; \chi, \psi) .
  \]
\end{proposition}

\begin{proof}
  Temporarily, denote with
  \[ p_s (X) = (X - 1)^{k - s - 1} (1 - X^{s - 1}) \]
  one of the terms in (\ref{eq:rx}) contributing to $R_k (X ; \chi, \psi)$. It
  is simple to check that
  \[ X^{k - 2} p_s (X^{- 1}) = (- 1)^{k - s} p_s (X) . \]
  On the other hand, recall that, for any Dirichlet character $\chi$, $B_{s,
  \chi} = 0$ if $\chi (- 1) \neq (- 1)^s$, unless $\chi = 1$ and $s = 1$. It
  follows that $B_{k - s, \psi} = 0$ if $\psi (- 1) \neq (- 1)^{k - s}$,
  unless $\psi = 1$ and $s = k - 1$. In the latter case, when $\psi = 1$ and
  $s = k - 1$, we have $B_{s, \chi} = B_{k - 1, \chi} = 0$ because $\chi (- 1)
  = (- 1)^k$, unless $\chi = 1$ and $k = 2$, which may be checked separately.
  We have thus shown that $X^{k - 2} p_s (X^{- 1}) = \psi (- 1) p_s (X)$ for
  all $s$ that have a nonzero coefficient in (\ref{eq:rx}).
\end{proof}

\begin{example}
  \label{eg:LasEichler}The case $\psi = 1$ is of special interest, because it
  yields explicit formulas for Dirichlet $L$-values at integral arguments (of
  parity opposite to the Dirichlet character) in terms of Eichler integrals.
  Let $\chi$ be a primitive Dirichlet character and $k \geqslant 3$ such that
  $\chi (- 1) = (- 1)^k$. Applying Theorem \ref{thm:ppE} as in Corollary
  \ref{cor:EtildaB} yields
  \begin{eqnarray*}
    \frac{k - 1}{2 \pi i}  \tilde{E}_k (X ; \chi, 1) |_{2 - k} (1 - R^L) & = &
    G \left( \chi \right) (- 2 \pi i / L)^{k - 1} R_k (L X + 1 ; \bar{\chi},
    1)\\
    &  & + L (k - 1, \chi) \left( 1 - \left( L X + 1 \right)^{k - 2} \right)
    .
  \end{eqnarray*}
  Solving for $L (k - 1, \chi)$, we obtain formulas for these $L$-values in
  the spirit of {\cite{katayama-L74}}; see Example \ref{eg:katayama}. On the
  other hand, suppose that $\alpha$, with $\tmop{Im} (\alpha) > 0$, is a root
  of $R_k (\alpha ; \bar{\chi}, 1) = 0$, which is not a $(k - 2)$th root of
  unity. Then
  \begin{equation}
    L (k - 1, \chi) = \frac{k - 1}{2 \pi i (1 - \alpha^{k - 2})}  \left[
    \tilde{E}_k \left( \frac{\alpha - 1}{L} ; \chi, 1 \right) - \alpha^{k - 2}
    \tilde{E}_k \left( \frac{1 - 1 / \alpha}{L} ; \chi, 1 \right) \right],
    \label{eq:LasEichler}
  \end{equation}
  thus explicitly linking the $L$-value to values of the Eichler integral at
  algebraic points, as is studied for $\chi = 1$ in
  {\cite{gun-murty-rath2011}}. Note that by (\ref{eq:Etildaq}), as in Example
  \ref{eg:katayama},
  \[ \tilde{E}_k (\tau ; \chi, 1) = \frac{4 \pi i}{k - 1} \sum_{n =
     1}^{\infty} \frac{\chi (n)}{n^{k - 1}}  \frac{e^{2 \pi i n \tau}}{e^{2
     \pi i n \tau} - 1} . \]
  Hence (\ref{eq:LasEichler}) takes the entirely explicit form
  \[ L (k - 1, \chi) = \frac{2}{1 - \alpha^{k - 2}}  \sum_{n = 1}^{\infty}
     \frac{\chi (n)}{n^{k - 1}}  \left[ \frac{1}{1 - e^{2 \pi i n (1 - \alpha)
     / L}} - \frac{\alpha^{k - 2}}{1 - e^{2 \pi i n (1 / \alpha - 1) / L}}
     \right], \]
  which expresses the $L$-value as a combination of two Lambert-type series.
  The special case $\chi = 1$, which is a consequence of Ramanujan's identity
  (\ref{eq:rhohat1}), has recently been studied in
  {\cite{gun-murty-rath2011}}. Specifically, using the notation
  \[ F_k (z) \assign \frac{k i}{4 \pi} \tilde{E}_{k + 1} (z ; 1, 1) = \sum_{n
     = 1}^{\infty} \frac{1}{n^k}  \frac{e^{2 \pi i n \tau}}{1 - e^{2 \pi i n
     \tau}}, \]
  it is shown in {\cite[Theorem 1.1]{gun-murty-rath2011}} that the numbers
  \[ F_{2 k + 1} (\beta) - \beta^{2 k} F_{2 k + 1} (- 1 / \beta) = \frac{(2 k
     + 1) i}{4 \pi} \left[ \tilde{E}_{2 k + 2} (\beta ; 1, 1) - \beta^{2 k}
     \tilde{E}_{2 k + 2} (- 1 / \beta ; 1, 1) \right] \]
  are transcendental for every algebraic $\beta \in \mathcal{H}$ with at most
  $2 k + 2 + \delta$ exceptions. Here $\delta = 0, 1, 2, 3$ according to
  whether $\gcd (k, 6) = 1, 2, 3$ or $6$, respectively. Examples of such
  exceptional values are $\beta = i$, if $k$ is even, and $\beta = e^{\pi i /
  3}, e^{2 \pi i / 3}$, if $k$ is divisible by $3$. In each of these three
  cases, however, the combination $F_{2 k + 1} (\beta) - \beta^{2 k} F_{2 k +
  1} (- 1 / \beta)$ vanishes. As pointed out in {\cite{gun-murty-rath2011}},
  the existence of further exceptional $\beta$ is very unlikely and would
  imply that $\zeta (2 k + 1)$ is an algebraic linear combination of $1$ and
  $\pi^{2 k + 1}$. This conclusion relies on the results of
  {\cite{msw-zeroes}}, by which the only roots of unity that are zeros of the
  corresponding period polynomial are $\pm i$, if $k$ is even, and $\pm e^{\pi
  i / 3}, \pm e^{2 \pi i / 3}$, if $k$ is divisible by $3$. The above results
  indicate that similar transcendence results should hold for the more general
  Eichler integrals $\tilde{E}_k (\tau ; \chi, 1)$ in place of $\tilde{E}_k
  (\tau ; 1, 1)$, in which case the arithmetic nature of Dirichlet $L$-values
  of ``wrong'' parity is concerned.
\end{example}

\begin{remark}
  It appears that the observations in this section can be extended further.
  For instance, in the case of imaginary Dirichlet characters, one finds that
  most of the zeros lie on the unit circle, if one considers instead of $R_k
  (X ; \chi, \psi)$ its real or imaginary part (that is, the polynomial with
  coefficients which are the real or imaginary parts of the coefficients of
  $R_k (X ; \chi, \psi)$). Finally, generalizing (\ref{eq:rx}), one may
  consider the polynomial
  \[ \sum_{s = 0}^k \frac{B_{s, \chi}}{s!} \frac{B_{k - s, \psi}}{\left( k - s
     \right) !} \left( \frac{X - 1}{m} \right)^{k - s - 1} (1 - X^{s - 1}), \]
  where $m = a M$ for some integer $a$. In the spirit of Example
  \ref{eg:rperiod}, this polynomial corresponds to the period polynomial in
  Corollary \ref{cor:EtildaB} when $n = a L$. It again appears that most roots
  of these polynomials lie on the unit circle.
\end{remark}

As a second, and possibly more direct, generalization of the Ramanujan
polynomials (\ref{eq:RkX}), we briefly also consider
\begin{equation}
  S_k (X ; \chi, \psi) = \sum_{s = 0}^k \frac{B_{s, \chi}}{s!} \frac{B_{k - s,
  \psi}}{(k - s) !} \left( \frac{L X}{M} \right)^{k - s - 1}, \label{eq:sx}
\end{equation}
where $\chi$ and $\psi$ are Dirichlet characters modulo $L$ and $M$. The two
generalizations (\ref{eq:rx}) and (\ref{eq:sx}) are related by
\[ S_k (X ; \chi, \psi) |_{2 - k} (1 - R^L) = R_k (L X + 1 ; \chi, \psi) . \]
By (\ref{eq:rhohatEtilda}) and Corollary \ref{cor:rhohat}, we have that, for
$k \geqslant 3$ and primitive nonprincipal Dirichlet characters $\chi$ and
$\psi$ such that $\chi (- 1) \psi (- 1) = (- 1)^k$,
\[ \tilde{E}_k (X ; \chi, \psi) - \psi (- 1) X^{k - 2} \tilde{E}_k (- 1 / X ;
   \psi, \chi) = - \chi (- 1) G \left( \chi \right) G (\psi) \frac{(2 \pi i /
   L)^k}{k - 1} S_k (X ; \bar{\chi}, \bar{\psi}), \]
which expresses the $S_k (X ; \chi, \psi)$ as period polynomials as well. In
general, these polynomials are not self-inverse and therefore cannot be
unimodular.

\begin{conjecture}
  For nonprincipal real Dirichlet characters $\chi$, all nonzero roots of the
  polynomial $S_k (X ; \chi, \chi)$ lie on the unit circle.
\end{conjecture}

We have verified this conjecture numerically for $k \leqslant 50$ and
characters $\chi$ of modulus at most $100$. In the case $\chi = 1$, we recall
that it was shown in {\cite{msw-zeroes}} that the polynomials $S_{2 k} (X ; 1,
1) = R_{2 k} (X)$ have all their nonreal zeroes on the unit circle. On the
other hand, for $\chi = 1$, the polynomial $S_{2 k} (X ; 1, 1)$ is only the
odd part of the period polynomial and it was recently proved in
{\cite{ls-roots12}} that the full period polynomial (\ref{eq:rhohat1}) is
indeed unimodular.

\section{Relation to sums considered by Ramanujan}\label{sec:history}

\subsection{Sums of level 1}\label{sec:rama1}

Ramanujan famously recorded (see {\cite{berndt-modular77}} or {\cite[p.
276]{berndtII}}, as well as the references therein) the formula
\begin{eqnarray}
  &  & \alpha^{- m} \left\{ \frac{\zeta (2 m + 1)}{2} + \sum_{n = 1}^{\infty}
  \frac{n^{- 2 m - 1}}{e^{2 \alpha n} - 1} \right\} = (- \beta)^{- m} \left\{
  \frac{\zeta (2 m + 1)}{2} + \sum_{n = 1}^{\infty} \frac{n^{- 2 m - 1}}{e^{2
  \beta n} - 1} \right\} \nonumber\\
  &  & - 2^{2 m} \sum_{n = 0}^{m + 1} (- 1)^n \frac{B_{2 n}}{(2 n) !} 
  \frac{B_{2 m - 2 n + 2}}{(2 m - 2 n + 2) !} \alpha^{m - n + 1} \beta^n, 
  \label{eq:rama}
\end{eqnarray}
where $\alpha$ and $\beta$ are positive numbers with $\alpha \beta = \pi^2$
and $m$ is any nonzero integer. Rewriting equation (\ref{eq:rama}) using
\[ \frac{1}{e^x - 1} = \frac{1}{2} \cot \left( \frac{x}{2} \right) -
   \frac{1}{2}, \]
slightly shifting the value of $m$ and setting $\beta = \pi i \tau$, and
therefore $\alpha = - \pi i / \tau$, we obtain, as in
{\cite{berndt-dedekindsum76}}, for integers $m \neq 1$,
\[ \tau^{2 m - 2} \xi_{2 m - 1} (- \tfrac{1}{\tau}) - \xi_{2 m - 1} (\tau) =
   (- 1)^m (2 \pi)^{2 m - 1} \sum_{n = 0}^m \frac{B_{2 n}}{(2 n) !} 
   \frac{B_{2 m - 2 n}}{(2 m - 2 n) !} \tau^{2 n - 1}, \]
where $\xi_s$ is the cotangent Dirichlet series defined in (\ref{eq:defcz}).

As in the case of the secant Dirichlet series, discussed in Sections
\ref{sec:eisenstein} and \ref{sec:periods}, the cotangent Dirichlet series is
essentially an Eichler integral. Indeed, one easily checks that, for integral
$s$,
\[ \frac{i}{2} \xi_s (\tau) = \frac{1}{2} \zeta (s) + \sum_{n = 1}^{\infty}
   \frac{\sigma_s (n)}{n^s} q^n . \]
When $s = 2 k - 1$ is odd, the right-hand side visibly is, up to scaling and a
constant term, an Eichler integral of the weight $2 k$ Eisenstein series
\[ E_{2 k} (\tau ; 1, 1) = 2 \zeta (2 k) + \frac{2 (2 \pi i)^{2 k}}{\Gamma (2
   k)} \sum_{n = 1}^{\infty} \sigma_{2 k - 1} (n) q^n, \]
which is modular with respect to the full modular group.

\begin{remark}
  Proceeding as in Section \ref{sec:sqrt}, though matters simplify because the
  level is $1$, we conclude that $\xi_{2 m + 1} ( \sqrt{r})$ is a rational
  multiple of $\pi^{2 m + 1} \sqrt{r}$ whenever $r$ is a positive rational
  number (assuming, for convergence, that $\sqrt{r}$ is irrational). Explicit
  special cases of this observation may be found, for instance, in
  {\cite{berndt-dedekindsum76}}.
\end{remark}

\subsection{Sums of level 4}

Ramanujan also found {\cite[Entry 21(iii), p. 277]{berndtII}} the identity
\begin{eqnarray}
  &  & \alpha^{- m + 1 / 2} \left\{ \tfrac{1}{2} L (\chi_{- 4}, 2 m) +
  \sum_{n = 1}^{\infty} \frac{\chi_{- 4} (n)}{n^{2 m} (e^{\alpha n} - 1)}
  \right\} = \frac{(- 1)^m \beta^{- m + 1 / 2}}{2^{2 m + 1}} \sum_{n =
  1}^{\infty} \frac{\tmop{sech} (\beta n)}{n^{2 m}} \nonumber\\
  &  & + \frac{1}{4} \sum_{n = 0}^m \frac{(- 1)^n}{2^{2 n}}  \frac{E_{2
  n}}{(2 n) !}  \frac{B_{2 m - 2 n}}{(2 m - 2 n) !} \alpha^{m - n} \beta^{n +
  1 / 2},  \label{eq:ramasec}
\end{eqnarray}
which was first proved in print by Chowla {\cite{chowla-series}}, where
$\alpha$ and $\beta$ are positive numbers with $\alpha \beta = \pi^2$ and $m$
is any integer. The goal of this section is to relate (\ref{eq:ramasec}) to
the present discussion of the modular properties of the secant Dirichlet
series. It will transpire that (\ref{eq:ramasec}) is an explicit version of
Theorem \ref{thm:ppE} in the case $\chi = 1$ and $\psi = \chi_{- 4}$. In other
words, equation (\ref{eq:ramasec}) encodes how $\psi_{2 m}$ transforms under
$S$, as defined in (\ref{eq:TS}). From here, we can then work out the exact
way in which $\psi_{2 m}$ transforms under any transformation of $\tmop{SL}_2
(\mathbbm{Z})$, though we do not develop the details here (but see Remark
\ref{rk:phifun23} for indications). On the other hand, we demonstrate that we
may use (\ref{eq:ramasec}) as the basis for yet another derivation of the
functional equation (\ref{eq:phifun}).

With $\alpha = \pi i \tau$, and proceeding as in the case of (\ref{eq:rama}),
we can write (\ref{eq:ramasec}) as
\begin{equation}
  \tau^{2 m - 1} \psi_{2 m} (- \tfrac{1}{\tau}) = \hat{\psi}_{2 m} (\tau) -
  h_{2 m} (\tau), \label{eq:phifunS}
\end{equation}
where
\[ \hat{\psi}_{2 m} (\tau) = \sum_{n = 1}^{\infty} \chi_{- 4} (n) \frac{\cot (
   \tfrac{\pi n}{2} \tau)}{(n / 2)^{2 m}} \]
and $h_{2 m} (\tau)$ is the rational function defined in (\ref{eq:def-h}).
Since $\hat{\psi}_{2 m} (\tau + 2) = \hat{\psi}_{2 m} (\tau)$, we find that
\[ \psi_{2 m} |_{1 - 2 m} S T^{- 2} = \hat{\psi}_{2 m} - h_{2 m} |_{1 - 2 m}
   T^{- 2} . \]
Using further that $B = S T^{- 2} S^{- 1}$, with $B$ as in (\ref{eq:AB}), we
find that
\begin{eqnarray}
  \psi_{2 m} |_{1 - 2 m} B & = & \psi_{2 m} - h_{2 m} |_{1 - 2 m} (S^{- 1} +
  T^{- 2} S^{- 1}) \nonumber\\
  & = & \psi_{2 m} + h_{2 m} |_{1 - 2 m} (S + T^{- 2} S) . 
  \label{eq:phifun-h}
\end{eqnarray}
which, in expanded form, is precisely the functional equation
(\ref{eq:phifun2}).

\begin{remark}
  \label{rk:phifun23}The functional equation satisfied by $\psi_{2 m}$ is
  given in {\cite{lrr-secantzeta}} in somewhat different form. Indeed,
  comparing the rational functions involved in the functional equations, we
  observe that
  \begin{eqnarray}
    &  & (\pi i)^{2 m} \sum_{n = 0}^m \frac{2^{2 n - 1} B_{2 n} E_{2 m - 2
    n}}{(2 n) ! (2 m - 2 n) !} \tau^{2 m - 2 n} \left[ 1 - (2 \tau + 1)^{2 n -
    1} \right] \nonumber\\
    & = & (\pi i)^{2 m} \sum_{n = 0}^m \frac{2^{2 n - 1} B_{2 n} (
    \tfrac{1}{2}) E_{2 m - 2 n}}{(2 n) ! (2 m - 2 n) !}  (\tau + 1)^{2 m - 2
    n} \left[ 1 - \left( 2 \tau + 1 \right)^{2 n - 1} \right], 
    \label{eq:phifun-rats}
  \end{eqnarray}
  where the first sum is the right-hand side of (\ref{eq:phifun2}) and the
  second sum is, up to notation, the one derived in {\cite{lrr-secantzeta}}.
  To see that these two rational functions indeed coincide, one may proceed as
  in Remark \ref{rk:phifun-12}. On the other hand, in order to further
  illustrate how $\psi_{2 m}$ transforms under the full modular group, we now
  sketch a proof of their equivalence, which, ultimately, derives from $B = S
  T^{- 2} S^{- 1} = (T S T)^2$.
  
  We can easily check that $\psi_{2 m} |_{1 - 2 m} T$ is given by
  \[ \psi_{2 m} (\tau + 1) = \sum_{n = 1}^{\infty} (- 1)^n \frac{\sec (\pi n
     \tau)}{n^{2 m}} = \frac{1}{2^{2 m - 1}} \psi_{2 m} (2 \tau) - \psi_{2 m}
     (\tau) . \]
  It follows that
  \[ \psi_{2 m} |_{1 - 2 m} T S = \hat{\psi}_{2 m} (\tau / 2) - \hat{\psi}_{2
     m} (\tau) - h_{2 m} (\tau / 2) + h_{2 m} (\tau) . \]
  Since $\cot \left( z / 2 \right) - \cot \left( z \right) = \csc (z)$, this
  equals
  \[ \psi_{2 m} |_{1 - 2 m} T S = \sum_{n = 1}^{\infty} \chi_{- 4} (n)
     \frac{\csc ( \tfrac{\pi n}{2} \tau)}{(n / 2)^{2 m}} - h_{2 m} (\tau / 2)
     + h_{2 m} (\tau), \]
  which then implies that
  \begin{eqnarray*}
    \psi_{2 m} |_{1 - 2 m} T S T^2 & = & - \sum_{n = 1}^{\infty} \chi_{- 4}
    (n) \frac{\csc ( \tfrac{\pi n}{2} \tau)}{(n / 2)^{2 m}} - h_{2 m} (\tau /
    2 + 1) + h_{2 m} (\tau + 2)\\
    & = & - \psi_{2 m} |_{1 - 2 m} T S + [h_{2 m} (\tau) - h_{2 m} (\tau /
    2)] |_{1 - 2 m} (1 + T^2) .
  \end{eqnarray*}
  Since $B = (T S T)^2 = (T S T^2) S T$, we thus arrive at
  \begin{equation}
    \psi_{2 m} |_{1 - 2 m} B = \psi_{2 m} + [h_{2 m} (\tau) - h_{2 m} (\tau /
    2)] |_{1 - 2 m} (S T + T^2 S T) . \label{eq:phifun-h2}
  \end{equation}
  Comparing (\ref{eq:phifun-h2}) with (\ref{eq:phifun-h}), we have shown that
  \[ h_{2 m} (\tau) |_{1 - 2 m} (S + T^{- 2} S) = [h_{2 m} (\tau) - h_{2 m}
     (\tau / 2)] |_{1 - 2 m} (S T + T^2 S T), \]
  which, upon expanding and using the relation $B_n ( \tfrac{1}{2}) = - (1 -
  2^{1 - n}) B_n$, results in the desired equality (\ref{eq:phifun-rats}).
\end{remark}

Having thus come full circle, we close with remarking that a number of further
infinite sums, similar in shape to (\ref{eq:defsz}) and (\ref{eq:defcz}), with
trigonometric summands and modular properties are discussed in {\cite[Chapter
14]{berndtII}} and the references therein.

\section{Conclusion}

We have reviewed the well-known fact that, among similar sums, the cotangent
and secant Dirichlet series of appropriate parity are Eichler integrals of
Eisenstein series of level $1$ and $4$, respectively. Their functional
equations, recorded by Ramanujan in his notebooks, are thus instances of the
modular transformation properties of Eichler integrals in general, with their
precise form determined by the corresponding period polynomials. This has lead
us to explicitly compute period polynomials of Eisenstein series of higher
level. Motivated by recent results {\cite{gun-murty-rath2011}},
{\cite{msw-zeroes}}, {\cite{lr-unimod13}}, {\cite{ls-roots12}} on the zeros of
Ramanujan polynomials, which arise in the level $1$ case, we observe that the
generalized Ramanujan polynomials appear to also be (nearly) unimodular. On
the other hand, it was recently shown in {\cite{cfi-ppz}} and
{\cite{er-ppz13}} that the nontrivial zeros of period polynomials of modular
forms, which are Hecke eigenforms of level $1$, all lie on the unit circle.
Our observations for Eisenstein series of higher level suggest that it could
be interesting to extend the results in {\cite{cfi-ppz}}, {\cite{er-ppz13}} to
the case of higher level.

\begin{acknowledgements}
We thank Matilde Lal\'{\i}n, Francis Rodrigue and
Mathew Rogers for sharing the preprint {\cite{lrr-secantzeta}}, which
motivated the present work. We are very grateful to Alexandru Popa for making
us aware of the recent paper {\cite{popa-period}} and for very helpful
discussions, as well as to Bernd Kellner for comments on an earlier version of
this paper. Finally, the second author would like to thank the
Max-Planck-Institute for Mathematics in Bonn, where part of this work was
completed, for providing wonderful working conditions.
\end{acknowledgements}


\end{document}